\documentclass[11pt,a4paper]{amsart}

\usepackage{amsmath}
\usepackage{amssymb}
\usepackage{amsthm}
\usepackage[notref,notcite,final]{showkeys}
\usepackage{xcolor}
\usepackage[margin=1.2in]{geometry}
\usepackage{mathtools}
\usepackage{graphicx}
\usepackage{enumitem}
\usepackage[ruled]{algorithm2e}
\usepackage{soul}
\usepackage{subcaption}
\usepackage{hyperref}
\usepackage{nicematrix}
\usepackage{multirow}
\usepackage{mathrsfs}
\usepackage{cleveref}
\graphicspath{{figures/}}
\allowdisplaybreaks

\newcommand{\mf}[1]{\mathfrak{#1}}

\DeclareMathOperator*{\dif}{d\!}

\newtheorem{thm}{Theorem}[section]

\newtheorem{lemma}[thm]{Lemma}
\newtheorem{prop}[thm]{Proposition}

\newtheorem{remark}[thm]{Remark}

\author[Yang]{Yang Yang}
\address{Department of Computational Mathematics Science and Engineering, Michigan State University, East Lansing, MI 48824, USA}
\curraddr{}
\email{yangy5@msu.edu}

\title[A Formula for Time-to-Frequency Wave Boundary Data Conversion]{A Formula for Time-to-Frequency Wave Boundary Data Conversion by the Boundary Control Method.
}

\begin{document}

\begin{abstract}
Given the wave equation on a compact Riemannian manifold with boundary, we derive an explicit reconstruction procedure to represent the frequency-domain Neumann-to-Dirichlet map in terms of the time-domain Neumann-to-Dirichlet map at any non-eigenfrequency. If the wave equation is exactly controllable, we derive an explicit formula to compute the former from the latter. The derivation is based on the boundary control method and requires only knowledge on the boundary of the manifold. The formula is stable when the level of regularization is fixed. The numerical feasibility is validated using one-dimensional examples in both Euclidean and non-Euclidean geometries.
\end{abstract}

\maketitle

\section{Introduction}
Inverse problems for the wave equation arise in numerous tomographic imaging applications. They may be posed either in the frequency domain or in the time domain. For the former, tools from elliptic inverse theory are natural; while for the latter, hyperbolic and control-theoretic techniques are often used. 
While classical transforms (e.g, Laplace/Fourier transform) formally connect time-domain and frequency-domain data, these routes typically require measurements for all time or analytic continuation, and thus do not directly yield finite-time, numerically stable representations.
In this paper, we derive an explicit formula that converts time-domain boundary data into frequency-domain boundary data without using the temporal Fourier or Laplace transform. This numerically feasible formula enables multi-frequency elliptic methods to be applied directly to time-domain data.

Throughout the paper, $(M,g)$ denotes an $n$-dimensional smooth compact Riemannian manifold with boundary $\partial M$. Let $\mf{P}$ be the following second-order elliptic operator on $M$:
$$
\mf{P}u := - c^2 \Delta_g u + qu
$$
where $0< c \in C^\infty(M)$ is a real-valued smooth positive \textit{wave speed}, $q\in C^\infty(M)$ is a real-valued smooth \textit{potential}, and $\Delta_g$ is the Laplace-Beltrami operator with coordinate representation
$$
\Delta_g u = \frac{1}{\sqrt{\det g}} \sum^n_{j,k=1} \partial_j \left( \sqrt{\det g} g^{jk} \partial_k u \right).
$$
We proceed to define the boundary data in the frequency domain and time domain.

\subsection{Inverse Boundary Value Problems in the Frequency Domain}
If $\lambda\in\mathbb{C}$ is not a Neumann eigenvalue of $\mf{P}$ and $\mf{f}\in H^{-\frac{1}{2}}(M)$, the elliptic boundary value problem
$$
(\mf{P} - \lambda) \mf{u}_\lambda = 0 \; \text{ on } M, \qquad \partial_\nu \mf{u}_\lambda|_{\partial M} = \mf{f}
$$
admits a unique weak solution $\mf{u}_\lambda = \mf{u}^{\mf{f}}_\lambda\in H^1(M)$. Here $\partial_\nu$ is the boundary normal derivative with respect to the metric $g$ which in local coordinates is $\partial_\nu u = \sum^n_{j,k=1} g^{jk} \nu_j \partial_k u$.
The \textit{elliptic Neumann-to-Dirichlet map (ND map) at frequency $\lambda$} is defined as 
\begin{equation} \label{eq:ellipticNDmap}
\mathfrak{L}_\lambda: \mf{f} \mapsto \mf{u}_\lambda|_{\partial M}, \qquad \lambda\in\mathbb{C} \text{ not a Neumann eigenvalue.}
\end{equation}
The linear map $\mathfrak{L}_\lambda: H^{-\frac{1}{2}}(M) \to H^{\frac{1}{2}}(M)$ is bounded by the standard elliptic regularity theory. It encodes the correspondence between the elliptic Neumann and Dirichlet data at frequency $\lambda$. We view $\mf{L}_\lambda$ as the frequency-domain wave boundary data.

Recovering one or more of the coefficients $(c,g,q)$ from knowledge of the ND map $\mf{L}_\lambda$ at one or multiple non-eigenfrequencies $\lambda$ constitutes a fundamental class of inverse problems, commonly referred to as \textit{inverse boundary value problems in the frequency domain}. These problems aim to determine interior analytical/geometric/topological properties of a manifold from boundary measurements and play a crucial role in inverse scattering and spectral analysis. A prominent special case is the Calder\'on's problem, which corresponds to replacing the ND map by the Dirichlet-to-Neumann (DN) map and considering the zero-frequency regime $\lambda=0$ (in which case Neumann eigenvalues are replaced by Dirichlet eigenvalues). The Calder\'on's problem has been extensively studied in the literature. We refer readers to the survey paper~\cite{uhlmann2009electrical} for detailed discussion and more comprehensive results.

\medskip
\subsection{Inverse Boundary Value Problems in the Time Domain}
The initial-boundary value problem associated with the wave dynamics of the elliptic operator $\mf{P}$ is
\begin{equation} \label{eq:ibvp}
\left\{
\begin{array}{rcl}
\partial^2_t u(t,x) + \mf{P}u & = & 0, \quad\quad\quad \text{ in } (0,2T) \times M \\
 \partial_\nu u  & = & f, \quad\quad\quad \text{ on } (0,2T) \times \partial M \\ 
 u(0,x) = \partial_t u(0,x) & = & 0 \quad\quad\quad\quad x \in M.
\end{array}
\right.
\end{equation}
where $T>0$ is a constant. Given $f\in C^\infty_c((0,2T)\times\partial\Omega)$, the well-posedness theory for the wave equation ensures a unique solution $u(t,x)=u^f(t,x)$. We define the \textit{hyperbolic Neumann-to-Dirichlet (ND) map} as:
\begin{equation} \label{eq:waveNDmap}
\Lambda f := u^f|_{(0,2T)\times \partial M}.
\end{equation}
The linear map $\Lambda$ extends to a bounded linear operator on $L^2((0,2T)\times\partial M)$, see Lemma~\ref{thm:operators} below. It encodes the correspondence between the hyperbolic Neumann and Dirichlet data. We view $\Lambda$ as the time-domain wave boundary data.

Recovering one or more of the coefficients $(c,g,q)$ from knowledge of the ND map $\Lambda$ constitutes another fundamental class of inverse problems, known as \textit{inverse boundary value problems in the time domain}.
These problems play a crucial role in scattering theory and wave-based imaging techniques. Note that if the observation time is finite, that is $T<\infty$, the finite speed of propagation for the wave equation gives a lower bound for $T$ in order for inverse boundary value problems in the time domain to have a unique solution. To see this, let us define
$$
T^* := \max_{x\in M} d(x,\partial M).
$$
where $d$ is the Riemannian distance function of the Riemannian manifold $(M,c^{-2} g)$. If $T<d(x_0,\partial M)$ for some $x_0\in M$, then the boundary measurement $\Lambda$ does not contain any information near $x_0$. Conversely, the inverse boundary value problem in the time domain is known to be uniquely solvable up to natural gauge transformations if $T>T^*$ using the boundary control method originated from~\cite{belishev1987approach} or complex geometric optics solutions originated from~\cite{sylvester1987global}. We refer readers to the survey papers~\cite{belishev2017boundary,belishev2007recent} and the monograph~\cite{katchalov2001inverse} for detailed discussion.

\subsection{Equivalence of Frequency-Domain and Time-Domain Wave Boundary Data}
The frequency-domain wave boundary data~\eqref{eq:ellipticNDmap} and the time-domain wave boundary data~\eqref{eq:waveNDmap} are well known to be equivalent. For instance, Katchalov et al.~\cite{katchalov2004equivalence} established a general equivalence between several types of boundary and spectral data associated with evolution equations, including the wave, heat, and Schr\"odinger equations. As one consequence of this theory, the frequency-domain Dirichlet-to-Neumann (DN) map and the time-domain DN map uniquely determine each other. Although this approach is mathematically sound, it relies on analytic continuation of holomorphic functions arising from Laplace-transform–based representations, a procedure that is severely ill-posed and therefore challenging for numerical implementation. Moreover, the representation requires boundary measurements over infinite observation time $T=\infty$, which must be truncated in any numerical realization, introducing inevitable truncation errors. To the best of our knowledge, the representation derived in~\cite{katchalov2004equivalence} has not been implemented numerically in the literature.

In this paper, we take a different approach to derive an explicit procedure that reconstructs the frequency-domain ND map from the time-domain ND map involving no analytic continuation and only finite observation time $T<\infty$. If the wave equation is exactly controllable, the procedure can be turned into a formula for numerical computation. Our approach is based on the boundary control method (BC method), pioneered by Belishev~\cite{belishev1987approach}. The BC method provides a systematic framework for analyzing wave propagation from boundary measurements by exploiting controllability, observability, and time-reversal properties of the wave equation. Its central idea is that boundary excitations generate waves whose interior dynamics encode geometric and analytic information about the underlying domain and coefficients, and that this information can be accessed through suitably designed boundary controls and measurements. Combined with Tataru’s unique continuation result~\cite{tataru1995unique, Tataru99unique}, the BC method has proven to be a powerful tool for establishing identifiability of coefficients in evolution equations. We refer the reader to the survey and monograph~\cite{belishev2017boundary,belishev2007recent,katchalov2001inverse} for various applications of the BC method.

An advantage of the BC method is its constructive nature: it often leads to operator identities linking measured boundary data directly or indirectly to quantities of interest. For this reason, the BC method often leads to reconstruction procedure. For example, numerical implementations based on the BC method have been developed to reconstruct the wave speed $c$/metric $g$~\cite{belishev1999dynamical, belishev2016numerical, de2018recovery, korpela2018discrete, oksanen2024linearized,pestov2012numerical, pestov2010numerical, yang2021stable} and the potential $q$~\cite{oksanen2022linearized} for acoustic wave equations. Nevertheless, reconstructions obtained via the BC method are typically unstable, reflecting the ill-posedness inherent in inverse problems for wave equations. As a consequence, practical implementations require suitable regularization, and stability is usually obtained only in a conditional or regularized sense. 
The present work follows this paradigm: the BC method is used to derive an explicit representation formula, while stability is ensured by fixing the level of regularization, leading to a formulation that is both mathematically transparent and numerically feasible.

\subsection{The Contribution:} The paper concerns numerically-feasible reconstruction of the frequency-domain ND map $\mathfrak{L}_\lambda$ at any non-eigenfrequency $\lambda$ from the time-domain ND map $\Lambda$. The major contribution includes:
\begin{itemize}
    \item A reconstruction procedure with finite observation time $T<\infty$. Laplace-transform–based representations are naturally formulated using boundary measurements over infinite time $T=\infty$. However, results in wave inverse problems show that finite observation time -- often dictated by the geometry of the manifold $(M,c^{-2} g)$ -- is sufficient to uniquely determine the frequency-domain ND map. In contrast to infinite observation time, we derive a reconstruction procedure that requires only finite observation time $T > T^*$, see Theorem~\ref{thm:ReconProcedure}. The finite observation time is more compatible with both theoretical insights from wave propagation and numerical implementation constraints. A striking feature of the procedure is that the reduction from time-domain to frequency-domain data is carried out entirely on the boundary. No interior information about $c,g,q$ is required. The formula thus applies to general compact Riemannian manifolds with boundary.

    \item A reconstruction formula with finite observation time $T<\infty$ for exactly controllable wave equations. Theoretical reconstruction formulas based on the Laplace transform rely on analytic continuation for certain non-eigenfrequencies, which is known to be ill-posed and numerically challenging (see Section~\ref{sec:idea} for detailed discussion). In contrast, we derive an explicit reconstruction formula in Theorem~\ref{thm:ReconFormula} that applies uniformly to all Neumann non-eigenfrequencies, provided the wave equation is exactly controllable. The formula utilizes an explicit regularization procedure instead of analytic continuation, allowing direct reconstruction. 

    \item Numerical feasibility and stability. The reconstruction formula involves only bounded linear operators and inverse of regularized operators, making it well suited for numerical implementation. Furthermore, we prove a local Lipschitz-type stability estimate in Theorem~\ref{thm:stability} when the regularization parameter is fixed, providing theoretical justification for the numerical stability. The reconstruction algorithm is further validated through one-dimensional numerical examples in both Euclidean and non-Euclidean geometries to confirm its numerical feasiblity and stability.
\end{itemize}

\medskip
The paper is organized as follows. Section~\ref{sec:derivation} introduces the notation and establishes preliminary results needed for the subsequent analysis. Section~\ref{sec:idea} reviews the Laplace transform–based approach and presents the main ideas underlying the boundary control method adopted in this work. Section~\ref{sec:density} presents two density results under the assumption that the wave equation is approximately controllable and exactly controllable respectively.
In Section~\ref{sec:reconprocedure}, we derive the reconstruction procedure and formula. Section~\ref{sec:stability} proves a stability estimate for the reconstruction formula with a fixed regularization parameter. Finally, Section~\ref{sec:numerics} implements the reconstruction formula and evaluates its performance through numerical experiments in both Euclidean and non-Euclidean geometries.

\bigskip
\section{Notations and Preliminaries} \label{sec:derivation}

Given a function $u(t,x)$, we write $u(t)=u(t,\cdot)$ for the spatial part as a function of $x$. We denote the Riemannian volume form on $M$ by $dx$, and the induced measure on $\partial M$ by $d\mu_{\partial M}$.
Introduce the following notations for various $L^2$-inner products:
\begin{align*}
(U,V)_{M} := & \int_M U(x) V(x) \,dx, & U,V\in L^2(M) \\
(U,V)_{c^{-2},M} := & \int_M U(x) V(x) \, c^{-2} dx = (c^{-1}U, c^{-1}V)_M, & U,V\in L^2(M) \\
(U,V)_{\partial M} := & \int_{\partial M} U(x) V(x) \,d \mu_{\partial M}(x), & U,V\in L^2(\partial M) \\
(U,V)_{(0,T)\times\partial M} := & \int^T_0 \int_{\partial M} U(t,x) V(t,x) \,d \mu_{\partial M}(x) \,dt, & U,V\in L^2((0,T)\times\partial M)
\end{align*}
The norms induced by these inner products are denoted by $\|\cdot\|_M$, $\|\cdot\|_{c^{-2},M} = \|c^{-1}\cdot \|_M$, $\|\cdot\|_{\partial M}$ and $\|\cdot\|_{(0,T)\times\partial M}$, respectively.
Note that $c^{-1}\in C^\infty(M)$ is a positive smooth function bounded from above and below on $M$, hence the norms $\|\cdot\|_M$, $\|\cdot\|_{c^{-2},M}$ are equivalent.
For a positive integer $s$, define
\begin{align*}
    H^s_{00}((0,T)\times\partial M) := & \{f: \; f, \partial_t f, \dots, \partial^{s}_t f \in L^2((0,T)\times\partial M), \\
    & \; f|_{t=0} = \partial_t f|_{t=0} = \dots = \partial^{s-1}_t f|_{t=0} = 0 \}
\end{align*}
equipped with the norm
$$
\|f\|^2_{H^s_{00}((0,T)\times\partial M)} := \sum^s_{j=1} \| \partial^j_t f \|^2_{(0,T)\times\partial M}.
$$

Let us introduce a few linear operators to be used later. We write 
$$
P_T: L^2((0,2T)\times\partial M) \rightarrow L^2((0,T)\times\partial M)
$$
for the orthogonal projection via restriction. Its adjoint operator 
$$
P^\ast_T: L^2((0,T)\times\partial M) \rightarrow L^2((0,2T)\times\partial M)
$$
is the extension by zero from $(0,T)$ to $(0,2T)$. For any $f\in L^2((0,T)\times\partial M)$, define the \textit{control operator}
$$
W: L^2((0,T)\times\partial M) \rightarrow L^2(M), \qquad f\mapsto u^{P^*_T f}(T) 
$$
where $u^{P^*_T f}$ is the solution of the initial boundary value problem~\eqref{eq:ibvp} with the Neumann source $P^*_T f$. 
The trace operator $\mathcal{T}$ on the boundary is defined as:
$$
\mathcal{T}: H^1_{00}((0,T)\times\partial M) \rightarrow L^2(\partial M), \quad f\mapsto f(T).
$$
Finally, we introduce the operator $S$ by requiring
$$
S := \mathcal{T} \partial^{-2}_t
$$
where $\partial^{-2}_t = \partial^{-1}_t \circ \partial^{-1}_t$ with $\partial^{-1}_t f(t,x) := \int^t_0 f(\tau,x) \,d\tau$. 

\begin{remark} \label{thm:extensionremark}
The definition of $W$ is independent of how $f$ is extended from $(0,T)$ to $(0,2T)$. Indeed, if $Q^*_T: L^2((0,T)\times\partial M) \rightarrow L^2((0,2T)\times\partial M)$ is another extension operator such that $P_T Q^*_T f = f$ on $L^2((0,T)\times\partial\Omega)$, then 
$$
u^{Q^*_T f} = u^{P^*_T f} \quad \text{ on } (0,T)\times M
$$ 
as they solve the same initial boundary value problem on $(0,T)\times M$ with identical Neumann data $f$. Henceforth, we will simply write 
$Wf = u^f(T)$ for $f\in L^2((0,T)\times\partial M)$ for the brevity of notation, but with the understanding that $f$ is extended in some way from $(0,T)$ to $(0,2T)$.
From time to time, we need to compute the derivative $\partial^2_t u^{f} = u^{\partial^2_t f}$ for $f$ with sufficient regularity. This should be understood as 
$$
\partial^2_t u^{f} = \partial^2_t u^{Q^*_T f} = u^{\partial^2_t Q^*_T f}
$$
for some extension $Q^*_T$ that preserves the regularity of $f$.
\end{remark}

\medskip
The operators $P_T,P^*_T$ are clearly continuous on the given spaces. We show that all the other operators defined above are continuous as well when considered on suitable Hilbert spaces.

\begin{lemma} \label{thm:operators}
The following linear operators are continuous between the given Hilbert spaces:
    \begin{enumerate}
        \item[(i)] 
        $\partial^{-1}_t: \; L^2((0,T)\times\partial M) \rightarrow L^2((0,T)\times\partial M)$, \\
        $\partial^{-1}_t: \; L^2((0,T)\times\partial M) \rightarrow H^1_{00}((0,T)\times \partial M)$.
        \item[(ii)] $W: \; L^2((0,T)\times\partial M) \rightarrow L^2(M)$.
        \item[(iii)] $\Lambda: \; L^2((0,2T)\times\partial M) \rightarrow \; L^2((0,2T)\times\partial M)$.
        \item[(iv)] $\mathcal{T}: H^1_{00}((0,T)\times\partial M) \rightarrow L^2(\partial M)$.
        \item[(v)] $S: L^2((0,T)\times\partial M) \rightarrow L^2(\partial M)$. 
    \end{enumerate}
\end{lemma}

\begin{proof}
    (i) For any $f\in L^2((0,T)\times\partial M)$, the Cauchy-Schwarz inequality implies
    $$
    \left| \partial^{-1}_t f(t,z) \right|^2 = \left| \int^t_0 f(s,z) \,ds \right|^2 \leq t \int^t_0 |f(s,z)|^2 \,ds, \qquad t\in (0,T), \quad z\in\partial M.
    $$
    Therefore,
    \begin{align*}
        \| \partial^{-1}_t f \|^2_{(0,T)\times\partial M } \leq \int_{\partial M} \int^T_0 \left[ t \int^T_0 |f(s,z)|^2 \,ds \right] \, dt \, d\mu_{\partial M}(z) = \frac{T^2}{2} \| f \|^2_{(0,T)\times\partial M }.
    \end{align*}
This shows the continuity of $\partial^{-1}_t: L^2((0,T)\times\partial M) \rightarrow L^2((0,T)\times\partial M)$. Moreover, 
$$
\| \partial_t (\partial^{-1}_t f) \|^2_{(0,T)\times\partial M } = \| f \|^2_{(0,T)\times\partial M }
$$
and $\partial^{-1}_t f|_{t=0} = 0$. We thus conclude the continuity of $\partial^{-1}_t: L^2((0,T)\times\partial M) \rightarrow H^1_{00}((0,T) \times\partial M)$.

(ii) For any Neumann data $h\in L^2((0,2T)\times\partial M)$, the initial boundary value problem~\eqref{eq:ibvp} admits a unique solution $u^h$, and the solution operator 
\begin{equation} \label{eq:SolutionOperator}
L^2((0,2T)\times\partial M) \rightarrow C([0,2T]; H^{\frac{3}{5}-\epsilon}(M)), \qquad h\mapsto u^h
\end{equation}
is continuous for any $\epsilon>0$~\cite[Theorem A]{lasiecka1991regularity}. Let $\epsilon<\frac{3}{5}$, then $Wf:=u^{P^*_T f}(T)$ maps $f\in L^2((0,T)\times\partial M)$ continuously into $H^{\frac{3}{5}-\epsilon}(M)$, which is continuously embedded in $L^2(M)$.

(iii) The hyperbolic ND map $\Lambda$ is the composition of the solution operator~\eqref{eq:SolutionOperator} followed by the trace operator onto $(0,2T)\times\partial M$. As a result, $\Lambda$ maps $L^2((0,2T)\times\partial M)$ continuously into $H^{\frac{3}{5}-\frac{1}{2}-\epsilon}((0,2T)\times\partial M)$, which is continuously embedded in $L^2((0,2T)\times\partial M)$ for small $\epsilon>0$.

(iv) For any $f\in H^1_{00}((0,T)\times\partial M)$, we have
$$
\left| f(T,z) \right|^2 = \left| \int^T_0 \partial_t f(s,z) \,ds \right|^2 \leq T \int^T_0 |\partial_t f(s,z)|^2 \,ds, \qquad z\in\partial M
$$
by the Cauchy-Schwarz inequality. Hence
$$
\|\mathcal{T}f\|^2_{\partial M} = \int_{\partial M} |f(T,z)|^2 \,d\mu_{\partial M}(z) \leq T \int_{\partial M} \int^T_0 |\partial_t f(s,z)|^2 \,ds \,d\mu_{\partial M}(z) \leq T \|f\|^2_{H^1_{00}((0,T)\times\partial M)}.
$$
    
(v) Simply notice that $S=\mathcal{T}\circ\partial^{-1}_t\circ\partial^{-1}_t$ is the composition of continuous operators. 
\end{proof}

\begin{remark}
    The operator $W: L^2((0,T)\times\partial M) \rightarrow L^2(M), \; f\mapsto u^{P^*_T f}(T)$ is indeed compact in view of the compact Sobolev embedding $H^{\frac{3}{5}-\epsilon}(M)\hookrightarrow L^2(M)$.
\end{remark}

\medskip
Later, we will need the $L^2$-adjoint of $\partial_t^{-2}$ and $S$. We collect the results in the next lemma.

\begin{lemma} \label{thm:CalculatingAdjoint}
    \begin{enumerate}
        \item The adjoint of $\partial^{-1}_t: L^2((0,T)\times\partial M) \rightarrow L^2((0,T)\times\partial M)$ is
        $$
        Z: L^2((0,T)\times\partial M) \rightarrow L^2((0,T)\times\partial M), \qquad Zf = \int^T_t f(s) \,ds.
        $$
        As a result, the adjoint of $\partial^{-2}_t: L^2((0,T)\times\partial M) \rightarrow L^2((0,T)\times\partial M)$ is $Z^2$.
        \item The adjoint of $S: L^2((0,T)\times\partial M) \rightarrow L^2(\partial M)$ is
        $$
        S^*: L^2(\partial M) \rightarrow L^2((0,T)\times\partial M), \qquad S^* \mf{f} = (T-t)\mf{f}.
        $$
    \end{enumerate}
\end{lemma}

\begin{proof}
    (1) For any $f,h\in L^2((0,T)\times\partial M)$, we write $h = \partial_t \partial^{-1}_t h$ and integrate by parts in the temporal variable to get
    \begin{align*}
        (Z f, h)_{(0,T)\times\partial M} & = (Z f, \partial_t \partial^{-1}_t h)_{(0,T)\times\partial M} = (Zf,\partial^{-1}_t h)_{\partial M}|^T_0 - (\partial_t Z f, \partial^{-1}_t h)_{(0,T)\times\partial M} \\
        & = (f, \partial^{-1}_t h)_{(0,T)\times\partial M}
    \end{align*}
where we used the fact that $\partial^{-1}_t h|_{t=0} = Zf|_{t=T}=0$ and $\partial_t Zf = -f$. This shows the adjoint of $\partial^{-1}_t$ is $Z$. It follows that the adjoint of $\partial^{-2}_t = \partial^{-1}_t \circ \partial^{-1}_t$ is $Z^2$.

(2) We will verify that $S^*$ defined as above is indeed the adjoint of $S$. Take any $f\in L^2((0,T)\times\partial\Omega)$ and any $\mf{f}\in L^2(\partial\Omega)$, then
$$
(f, S^* \mf{f})_{(0,T)\times\partial M} = (f, (T-t) \mf{f})_{(0,T)\times\partial M} = \int_{\partial M} \left[ \int^T_0 f(t) (T-t) \,dt \right] \; \mf{f} d\mu_{\partial M}.
$$
Let $h:=\partial^{-2}_t f$, then $\partial^2_t h = f$ and $h|_{t=0} = \partial_t h|_{t=0} = 0$. Integrate by parts twice to get
$$
\int_{\partial M} \left[ \int^T_0 \partial^2_t h(t) (T-t) \,dt \right] \mf{f} d\mu_{\partial M} = \int_{\partial M} \left[ \int^T_0 \partial_t h(t)\,dt \right] \, \mf{f} d\mu_{\partial M} = \int_{\partial M} h(T) \mf{f} \,d\mu_{\partial M} = (Sf, \mf{f})_{\partial M}.
$$
This is the relation that characterizes the adjoint of $S$.
\end{proof}

\bigskip
\section{The Idea} \label{sec:idea}

\subsection{A Theoretically Feasible Approach with $T=\infty$}

Given a function $u(t,x)$, its (temporal) Laplace transform is defined as
\begin{equation} \label{eq:LaplaceTransform}
\hat{u}(\omega,x) := \int^\infty_0 e^{-\omega t} u(t,x) \, dt. 
\end{equation}
We denote the open right half complex plane by 
$$
\mathbb{C}_R :=\{\omega\in\mathbb{C}: \Re\omega>0\},
$$
where $\Re\omega$ denotes the real part of $\omega$.

Fix $\omega\in\mathbb{C}$. For any $\mf{f}\in C^\infty(\partial M)$, pick $f \in C^\infty_c ((0,\infty)\times\partial M)$ such that $\hat{f}(\omega) = \mf{f}$. This can be done by choosing $\varphi \in C^\infty_c(0,\infty)$ so that $\int^\infty_0 e^{-\omega t} \varphi(t) \,dt=1$, then setting $f(t,x)=\varphi(t) \mf{f}(x)$.
Use such $f$ as the Neumann data to solve~\eqref{eq:ibvp} and let the solution be $u(t,x) = u(t)$. Since the energy of $u(t)$ is bounded for large $t$, the integral~\eqref{eq:LaplaceTransform} converges for $\omega\in\mathbb{C}_R$ and is an analytic function of $\omega$. For each $\omega\in\mathbb{C}_R$, $\hat{u}$ solves
$$
(\mf{P} - \omega^2) \hat{u}(\omega) = 0 \quad \text{ in } M, \qquad \partial_\nu \hat{u}(\omega)|_{\partial M} = \hat{f}(\omega)|_{\partial M} = \mf{f}|_{\partial M}.
$$
If $\omega^2$ is not a Neumann eigen-frequency, this boundary value problem has a unique solution $\hat{u}(\omega)$ with the Dirichlet data
$$
\mf{L}_{\omega^2} \mf{f} = \hat{u}(\omega)|_{\partial M} = \widehat{\Lambda f}(\omega)|_{\partial M}.
$$
This relation gives a theoretically feasible construction of $\mf{L}_{\omega^2}$ at any non-Neumann eigenfrequency $\lambda = \omega^2$ with $\omega\in\mathbb{C}_R$. We remark that the reverse representation, that is expressing $\Lambda$ in terms of $\mf{L}_{\lambda}$ with all non-Neumann eigenfrequency $\lambda$, is also possible, see~\cite[Equation (33)]{katchalov2004equivalence}.

While this approach is mathematically sound and theoretically feasible, it has clear limitations from the numerical perspective. First, the requirement $\omega\in\mathbb{C}_R$ means that the argument above does not apply to $\lambda<0$. This limitation can be remedied in theory by taking $\omega\in\mathbb{C}_R$ to approach $\sqrt{-\lambda} i$ for $\lambda<0$. However, the problem of recovering the boundary value of $\hat{u}(\omega)$ with $\omega$ on (or arbitrarily close to) the imaginary axis from the holomorphic function known in $\mathbb{C}_R$ is an analytic continuation problem, and analytic continuation problems are known to be severely ill-posed. As a result, numerical implementation for the case $\lambda<0$ is expected to be susceptible to noise in $\hat{u}(\omega)$ with $\omega\in\mathbb{C}_R$.

Second, the argument is based on infinite observation time $T=\infty$, which is theoretically permitted but numerically infeasible. Any numerical implementation has to truncate $T$ at a finite value, causing truncation errors. On the other hand, the requirement $T=\infty$ seems redundant when contrasted with the uniqueness results for wave inverse problems: It has been shown (e.g, see~\cite{belishev2007recent}) that finite observation time -- typically determined by the geometry of $(M,c^{-2} g)$ -- is sufficient to identify the operator $\mf{P}$ up to a natural gauge, which in turn is sufficient to determine $\mf{L}_\lambda$ for non-Neumann eigenfrequency $\lambda$.
It is therefore a natural attempt to reduce the infinite observation time to a finite value.

\subsection{A Numerically Feasible Approach with $T<\infty$}

In the rest of the paper, we develop a different approach based on the BC method to reconstruct $\mf{L}_\lambda$ for non-Neumann eigenfrequency $\lambda\in\mathbb{C}$.

The idea is as follows: for a given $\mf{f}\in L^2(\partial\Omega)$, we consider the minimization problem:
$$
\min_{f\in H^2_{00}} \|u^f_{tt}(T) + \lambda u^f(T) \|^2_{c^{-2},M} + \|f(T)-\mf{f}\|^2_{\partial M}
$$
where $\|\cdot\|_{c^{-2},M}$ is the norm induced by the inner product $(\cdot,\cdot)_{c^{-2},M} = (c^{-1} \cdot,c^{-1}\cdot)_{M}$. 
Here, the first term in the objective function promotes $u(T)$ to solve the elliptic boundary value problem, while the second term attempts to match the boundary condition.
Alternatively, we can use the isomorphism $\partial^{-2}_t: L^2((0,T)\times\partial M)\rightarrow H^2_{00}(0,T)\times\partial M)$ to write the optimization in terms of the variable $\ddot{f}:=\partial^2_t f$: 
$$
\min_{\ddot{f} \in L^2} \|u^{\ddot{f}}(T) + \lambda u^{\partial^{-2}_t \ddot{f}}(T)\|^2_{c^{-2},M} + \|\partial^{-2}_t \ddot{f}(T)-\mf{f}\|^2_{\partial M}.
$$
Using the operators $W,S,\partial^{-2}_t$, the optimization can be written as 
\begin{equation} \label{eq:keyopt}
\begin{aligned}
 \min_{\ddot{f} \in L^2} \|W \ddot{f} + \lambda W \partial^{-2}_t \ddot{f}\|^2_{c^{-2},M} + \|S \ddot{f}-\mf{f}\|^2_{\partial M}.
\end{aligned}
\end{equation}
This optimization can be viewed as the least-squares formulation of the linear system
\begin{equation} \label{eq:keyls}
\left(
\begin{array}{c}
     c^{-1}(W + \lambda W \partial^{-2}_t)  \\
     S 
\end{array}
\right) \ddot{f} = 
\left(
\begin{array}{c}
     0  \\
     \mf{f} 
\end{array}
\right) \qquad \text{ on } M \times\partial M.
\end{equation}

This motivates consideration of the following operator $A_\lambda$:
\begin{equation} \label{eq:Alambda}
A_\lambda:=\left(
\begin{array}{c}
     c^{-1} W (1 + \lambda  \partial^{-2}_t)  \\
     S 
\end{array}
\right).
\end{equation}
By Lemma~\ref{thm:operators}, all the operators in the definition of $A_\lambda$ are continuous. As $c\in C^\infty(M)$ is a smooth positive function bounded away from 0 and infinity, we conclude that 
$$
A_\lambda: L^2((0,T)\times\partial M) \rightarrow L^2(M) \times L^2(\partial M) 
$$ is continuous for any $\lambda\in\mathbb{C}$.
On the other hand, by using the natural embedding $L^{2}(M)\times L^{2}(\partial M) \hookrightarrow H^{-1}(M)\times H^{-\frac{1}{2}}(\partial M)$ where $H^s$ denotes the usual $L^2$-based Sobolev space of order $s\in\mathbb{R}$, we can also view $A_\lambda$ as the following continuous linear operator with the extended co-domain:
$$
A_\lambda: L^2((0,T)\times\partial M) \rightarrow H^{-1}(M)\times H^{-\frac{1}{2}}(\partial M).
$$

\bigskip
\section{Some Density Results} \label{sec:density}

A key observation in our proof is that $A_\lambda$ has a dense range under suitable assumptions. This section establishes two density results for $A_\lambda$ assuming approximate and exact controllability of the wave equation respectively.

First, we equip $A_\lambda$ with the extended co-domain $H^{-1}(M)\times H^{-\frac{1}{2}}(\partial M)$. Recall the following approximation controllability for the wave equation~\eqref{eq:ibvp}, which is a consequence of Tataru's unique continuation result~\cite{tataru1995unique, Tataru99unique}. Here, $H^1_0(M)$ is the Sobolev space of order 1 on $M$ with vanishing Dirichlet boundary conditions.
\begin{lemma}{\cite[Theorem 4.28]{katchalov2001inverse}} \label{thm:approxcontrol}
    For $T>2 T^*$, the linear subspace
    $
    \{(u^f(T), \partial_t u^f(T)): \; f\in C^\infty_c((0,T)\times\partial M)\}
    $
    is dense in $H^1_0(M)\times L^2(M)$.
\end{lemma}

Using Lemma~\ref{thm:approxcontrol}, we establish another approximate controllability result that enables approximation to functions in $H^1(M)$ with non-vanishing Dirichlet boundary conditions.

\begin{lemma} \label{thm:ourapproxcontrol}
    For $T>2T^*$, the linear subspace
    $$
    \{(u^f(T), \partial_t u^f(T)): \; f\in C_c^\infty((0,T]\times\partial M) \}
    $$
    is dense in $H^1(M) \times L^2(M)$.
\end{lemma}
\begin{proof}
Extend $M$ to a larger compact manifold with boundary $\tilde M$ such that $M$ is contained in the interior of $\tilde{M}$ and $T>2 \tilde{T}^*$, where $\tilde{T}^* := \max_{x\in \tilde{M}} d(x,\partial \tilde{M})$ is the critical time for the manifold $\tilde M$. Likewise, extend the differential operator $\mf{P}$ on $M$ to a smooth operator $\tilde{\mf{P}}$ on $\tilde M$.
For any $\epsilon>0$ and any $(\phi,\psi)\in H^1(M)\times L^2(M)$, we extend them to $(\tilde{\phi}, \tilde{\psi})\in H^1_0(\tilde M) \times L^2(\tilde{M})$. Lemma~\ref{thm:approxcontrol} ensures the existence of a function $h\in C^\infty_c((0,T)\times\partial \tilde{M})$ such that $\|\tilde{u}(T) - \tilde\phi\|_{H^1(\tilde{M})} + \|\partial_t \tilde{u}(T) - \tilde\psi\|_{L^2(\tilde{M})} < \epsilon$, where $\tilde{u}=\tilde{u}(t,x)$ solves
$$
\left\{
\begin{array}{rcl}
\partial^2_t \tilde{u} + \tilde{\mf{P}} \tilde{u} & = & 0 \quad\quad\quad \text{ in } (0,\infty) \times \tilde{M}, \\
 \partial_\nu \tilde{u}  & = & h \quad\quad\quad \text{ on } (0,\infty) \times \partial \tilde{M}, \\ 
 \tilde{u}(0,x) = \partial_t \tilde{u}(0,x) & = & 0 \quad\quad\quad\quad x \in \tilde{M}.
\end{array}
\right.
$$
Take $f:= \partial_\nu \tilde{u}|_{(0,T]\times\partial M}$, then $f$ is smooth (up to $t=T$) since $\tilde{u}$ is; $f$ is zero near $t=0$ since $h$ is compactly supported and wave propagates at finite speed. If $u=u^f(t,x)$ solves 
$$
\left\{
\begin{array}{rcl}
\partial^2_t u + \mf{P} u & = & 0 \quad\quad\quad \text{ in } (0,T) \times M, \\
 \partial_\nu u  & = & f \quad\quad\quad \text{ on } (0,T) \times \partial M, \\ 
 u(0,x) = \partial_t u(0,x) & = & 0 \quad\quad\quad\quad x \in M,
\end{array}
\right.
$$
then $u=\tilde{u}$ on $[0,T]\times M$ by the well-posedness of the initial boundary value problem. Therefore,
$$
\|u(T) - \phi\|_{H^1(M)} + \|\partial_t u(T) - \psi\|_{L^2(M)} \leq 
\|\tilde{u}(T) - \tilde\phi\|_{H^1(\tilde{M})} + \|\partial_t \tilde{u}(T) - \tilde\psi\|_{L^2(\tilde{M})} < \epsilon.
$$
This proves the desired density result.
\end{proof}

Using Lemma~\ref{thm:ourapproxcontrol}, we show that the range of $A_\lambda$ is dense in the extended co-domain.

\begin{prop} \label{thm:denserange}
Let $T>2T^*$ and $\lambda\in\mathbb{C}$ be a non-Neumann eigenvalue of $\mf{P}$. Then
$$
A_\lambda: L^2((0,T)\times\partial M) \rightarrow H^{-1}(M)\times H^{-\frac{1}{2}}(\partial M).
$$
has a dense range in $H^{-1}(M)\times H^{-\frac{1}{2}}(\partial M)$.
\end{prop}

\begin{proof}

For any $\mf{v}\in H^{-1}(M)$ and any $\mf{f} \in H^{-\frac{1}{2}}(\partial M)$, let $\phi = \phi_\lambda \in H^1(M)$ be the weak solution of the following elliptic BVP
$$
(\mf{P} - \lambda) \phi = -c \mf{v} \quad \text{ in } M, \qquad
\partial_\nu \phi|_{\partial M} = \mf{f}.
$$
By Proposition~\ref{thm:ourapproxcontrol}, there exists a sequence $f_j \in C^\infty_{c}((0,T]\times\partial\Omega)$ such that $Wf_j = u^{f_j}(T) \rightarrow \phi$ in $H^1(M)$ as $j\rightarrow\infty$. Therefore,
\begin{align*}
 c^{-1} W (1+\lambda \partial^{-2}_t) \ddot{f}_j & = c^{-1} \left( u^{\partial^2_t f_j}(T) + \lambda  W f_j \right) = c^{-1} \left( \partial^2_t u^{f_j}(T) + \lambda u^{f_j}(T) \right) \\ 
 & = - c^{-1} \left( \mf{P}u^{f_j}(T) - \lambda u^{f_j}(T) \right) \rightarrow -c^{-1} (\mf{P}-\lambda) \phi = \mf{v} \quad \text{ in } H^{-1}(M)
\end{align*}
and
$$
S\ddot{f}_j = f_j(T) = \partial_\nu u^{f_j}(T) \rightarrow \partial_\nu \phi = \mf{f} \quad \text{ in } H^{-\frac{1}{2}}(\partial M).
$$
as $j\rightarrow\infty$. Hence, any pair $(\mf{v}, \mf{f})\in H^{-1}(M)\times H^{-\frac{1}{2}}(\partial M)$ is in the closure of the range of $A_\lambda$. 

\end{proof}

Next, we equip $A_\lambda$ with the co-domain $L^2(M) \times L^2(\partial M)$. Let us first introduce some notations in order to discuss the exact controllability of the wave equation. As the operator $\mf{P}$ involves the wave speed $c$, it is natural to consider the travel time metric $c^{-2}g$ on $M$. Let $SM^\circ$ denote the unit sphere bundle over the interior of $(M, c^{-2}g)$, that is, $SM^\circ :=\{(x,v)\in SM: x\in M\setminus\partial M\}$ where $SM$ is the unit sphere bundle of $M$ with respect to the travel time metric $c^{-2}g$. 
For $(x,v)\in SM^\circ$, we denote by $\gamma_{x,v}(t)$ the unit speed geodesic on $M$ with $\gamma(0) = x$ and $\dot\gamma(0) = v$. The \textit{first exit time} of $(x,v)\in SM^\circ$ is defined as $\tau(x,v) := \inf\{t>0: \gamma_{x,v}(t)\in\partial M\}$ with the convention $\tau(x,v) = \infty$ if the geodesic $\gamma_{x,v}$ never meets $\partial M$.
Define
$$
\tau_{max} := \sup_{(x,v)\in SM^\circ} \tau(x,v).
$$
The manifold $(M,c^{-2}g)$ is said to be \textit{non-trapping} if $\tau_{max}<\infty$, in which case all the maximal geodesics starting in the interior hit $\partial M$ in finite time.

The exact controllability of the wave equation concerns the surjectivity of $W$. More precisely, the wave equation is said to be \textit{exactly controllable} at time $T>0$ if for a target function $\phi(x)$, there exists a Neumann source $f$ on $[0,T]\times\partial M$ such that $u^f(T)=\phi$. 
It is shown by Bardos, Lebeau and Rauch~\cite{bardos1992sharp} that the exact controllability property at time $T$ is equivalent to the Geometric Control Condition (GCC) , which in our setup asserts that all the rays of geometric optics in $M$ hit $\partial M$ in a time smaller than $T$, or equivalent, $T>\tau_{max}$. The exact controllability we need is the one with regularity proved in~\cite[Theorem 5.4]{ervedoza2010systematic}. The next lemma, shown in~\cite[Proposition 4]{oksanen2022linearized}, is a restatement of~\cite[Theorem 5.4]{ervedoza2010systematic} for a smooth $\phi$.

\begin{lemma}{\cite[Proposition 4]{oksanen2022linearized}} \label{thm:exactcontrol}
Suppose $T>\tau_{max}$. For any $\phi \in C^\infty(M)$, there is $f \in C_c^\infty((0,T] \times \partial M)$ such that $u^f(T) = \phi$ in $M$. Moreover, there is $C>0$, independent of $\phi$, such that 
$$
\|f\|_{H^2((0,T) \times \partial \Omega)} \le C \|\phi\|_{H^4(\Omega)}.
$$
\end{lemma}

Note that $T>\tau_{max}$ for some finite $T>0$ imposes further geometric constraints on $M$. In particular, $(M, c^{-2}g)$ is non-trapping. Using Lemma~\ref{thm:exactcontrol}, we show that the range of $A_\lambda$ is dense in $L^{2}(M)\times L^{2}(\partial M)$.

\begin{prop} \label{thm:L2density}
Let $T> \tau_{max}$ and $\lambda\in\mathbb{C}$ be a non-Neumann eigenvalue of $\mf{P}$. Then the range of 
$$
A_\lambda: L^2((0,T)\times\partial M) \to L^{2}(M)\times L^{2}(\partial M)
$$
is dense in $L^{2}(M)\times L^{2}(\partial M)$.
\end{prop}

\begin{proof}
For any $\mf{v}\in C^{\infty}(M)$ and any $\mf{f} \in C^{\infty}(\partial M)$, let $\phi_\lambda$ be the solution of the following elliptic BVP
$$
(\mf{P} - \lambda) \phi_\lambda = - c \mf{v} \quad \text{ in } M, \qquad
\partial_\nu \phi_\lambda|_{\partial M} = \mf{f}.
$$
Then $\phi_\lambda \in C^\infty(M)$ by the elliptic regularity. If $T> \tau_{max}$, by Lemma~\ref{thm:exactcontrol}, there exists $f\in C^\infty_c((0,T]\times\partial M)$ such that $u^f(T) = \phi_\lambda$. Hence
\begin{align*}
c^{-1} W (1+\lambda \partial^{-2}_t) \ddot{f} & = c^{-1} \left( u^{\partial^2_t f}(T) + \lambda  W f \right) = c^{-1} \left( \partial^2_t u^{f}(T) + \lambda u^{f}(T) \right) \\ 
& = - c^{-1} \left( \mf{P}u^{f}(T) - \lambda u^{f}(T) \right) = - c^{-1} (\mf{P}-\lambda) \phi_\lambda = \mf{v} 
\end{align*}
and
$$
S\ddot{f} = f(T) = \partial_\nu u^{f}(T) = \partial_\nu \phi_\lambda = \mf{f}. 
$$
This shows that the range of $A$ contains the set $C^{\infty}(M)\times C^{\infty}(\partial M)$, which is dense in $L^{2}(M)\times L^{2}(\partial M)$.
\end{proof}

\bigskip
\section{Reconstruction Procedure and Formula}
\label{sec:reconprocedure}

Recall the following standard result for the Tikhonov regularization (e.g, see \cite[Lemma 1]{oksanen2013solving}).
\begin{lemma}\label{thm:projection}
Let $A:X\rightarrow Y$ be a bounded linear operator between two Hilbert spaces $X$ and $Y$. For any $y\in Y$, let $\alpha>0$ be a constant and $x_\alpha := (A^\ast A + \alpha)^{-1} A^\ast y$. Then 
$$
A x_\alpha \rightarrow P_{\overline{Ran(A)}}y \quad\quad \text{ as } \alpha\rightarrow 0+
$$
where $P_{\overline{Ran(A)}}y$ denotes the orthogonal projection of $y$ onto the closure of the range of $A$.
\end{lemma}

Using this lemma, Proposition~\ref{thm:denserange} and Proposition~\ref{thm:L2density}, we can derive the following reconstruction procedure to obtain elliptic boundary data from the hyperbolic ND map $\Lambda$. Here, $A^*$ is the adjoint of $A$ equipped with suitable co-domains.

\begin{thm} \label{thm:ReconProcedure}
Let $\lambda$ be a non-Neumann eigenvalue of $\mf{P}$ and $f_{\lambda,\alpha}$ be a function with
$$
\ddot{f}_{\lambda,\alpha} := (A^*_\lambda A_\lambda + \alpha)^{-1} A^*_\lambda \left(
\begin{array}{c}
     0 \\
     \mf{f}
\end{array}
\right).
$$

\begin{enumerate}
    \item[(i)] If $T>2T^*$, then for any $\mf{f}\in H^{-\frac{1}{2}}(\partial M)$, we have
$$
\Lambda f_{\lambda,\alpha}(T) \rightarrow \mf{L}_\lambda \mf{f} \ \text{ in } H^{\frac{1}{2}}(\partial M) \quad \text{ as } \alpha\rightarrow 0+;
$$
    \item[(ii)] If $T>\tau_{\max}$,  then for any $\mf{f}\in L^{2}(\partial M)$, we have
$$
\Lambda f_{\lambda,\alpha}(T) \rightarrow \mf{L}_\lambda \mf{f} \ \text{ in } H^{1}(\partial M) \quad \text{ as } \alpha\rightarrow 0+.
$$
\end{enumerate}
\end{thm}

\begin{proof}
(i) In Lemma~\ref{thm:projection}, take $A=A_\lambda$, $X=L^2((0,T)\times\partial M)$ and $Y=H^{-1}(M) \times H^{-\frac{1}{2}}(\partial M)$. Note that $\overline{Ran(A)} = Y$ by Proposition~\ref{thm:denserange}. For any $(\mf{v}, \mf{f}) \in Y$, Lemma~\ref{thm:projection} asserts that if we define 
$$
\ddot{f}_{\lambda,\alpha} := (A^*_\lambda A_\lambda + \alpha)^{-1} A^*_\lambda \left(
\begin{array}{c}
     \mf{v} \\
     \mf{f}
\end{array}
\right) \quad \in X,
$$
then
$$
A_\lambda \ddot{f}_{\lambda,\alpha} = \left(
\begin{array}{c}
     c^{-1} W (1 + \lambda  \partial^{-2}_t)  \\
     S 
\end{array}
\right) \ddot{f}_{\lambda,\alpha} 
\rightarrow 
\left(
\begin{array}{c}
     \mf{v} \\
     \mf{f}
\end{array}
\right)
\quad \text{ in } Y \text{ as } \alpha \rightarrow 0+.
$$
In particular, choosing $\mf{v}=0$ gives
\begin{align*}
    c^{-1} W (1+\lambda \partial^{-2}_t) \ddot{f}_{\lambda,\alpha} & = c^{-1} \left( u^{\partial^2_t {f}_{\lambda,\alpha}}(T) + \lambda u^{f_{\lambda,\alpha}}(T) \right) = - c^{-1} (\mf{P}-\lambda) u^{{f}_{\lambda,\alpha}}(T) \rightarrow 0 \quad \text{ in } H^{-1}(M), \\
    S \ddot{f}_{\lambda,\alpha} & = f_{\lambda,\alpha}(T) = \partial_\nu u^{{f}_{\lambda,\alpha}}(T) \rightarrow \mf{f} \quad \text{ in } H^{-\frac{1}{2}}(\partial M)
\end{align*}
as $\alpha\rightarrow 0+$. As $c^{-1}>0$ is bounded from above and below, we conclude $u^{f_{\lambda,\alpha}}(T) \rightarrow \mf{u}_\lambda$ in $H^1(M)$ by the standard elliptic regularity estimate, where $\mf{u}_\lambda$ is the weak solution of the elliptic boundary value problem
$$
(\mf{P}-\lambda) \mf{u}_\lambda = 0 \quad \text{ in } M, \qquad \mf{u}_\lambda|_{\partial M} = \mf{f}.
$$
Restricting to $\partial M$ gives
$$
\Lambda f_{\lambda,\alpha}(T) = u^{f_{\lambda,\alpha}}(T)|_{\partial M} \rightarrow \mf{u}_\lambda|_{\partial M} = \mf{L}_\lambda \mf{f} \quad \text{ in } H^{\frac{1}{2}}(\partial M) \text{ as } \alpha\rightarrow 0+. 
$$

(ii) The proof goes the same way. This time, take $A=A_\lambda$, $X=L^2((0,T)\times\partial M)$ and $Y=L^{2}(M) \times L^{2}(\partial M)$ in Lemma~\ref{thm:projection}. We still have $\overline{Ran(A)} = Y$ by Proposition~\ref{thm:L2density}.
The convergence $A_\lambda \ddot{f}_{\lambda,\alpha}\rightarrow (0,\mf{f})^\top$ as $\alpha\to 0+$ along with the elliptic regularity implies that $u^{f_{\lambda,\alpha}}(T) \rightarrow \mf{u}_\lambda$ in $H^\frac{3}{2}(M)$. Restricting to $\partial M$ gives $\Lambda f_{\lambda,\alpha}(T) \rightarrow \mf{L}_\lambda \mf{f}$ in $H^{1}(\partial M)$ as $\alpha\rightarrow 0+$.

\end{proof}

Theorem~\ref{thm:ReconProcedure} suggests a path toward constructing the elliptic ND map $\mf{L}_\lambda$, as long as we can find $f_{\lambda,\alpha}$ from the hyperbolic ND map $\Lambda$ such that
$$
    (A_\lambda^*A_\lambda + \alpha I) \ddot{f}_{\lambda,\alpha} = A_\lambda^*
\left(
\begin{array}{c}
     0 \\
     \mf{f} 
\end{array}
\right).
$$
Using the definition of $A_\lambda$ (see~\eqref{eq:Alambda}), this linear equation becomes
\begin{equation} \label{eq:TikUsingWS}
[(1+\lambda \partial^{-2}_t)^* W^* c^{-2} W (1+\lambda \partial^{-2}_t) + S^* S + \alpha I] \ddot{f}_{\lambda,\alpha} = S^* \mf{f}.
\end{equation}
The rest of this section aims to explicitly represent the operators in this linear equation using the hyperbolic ND map $\Lambda$. From now on, we view $A_\lambda$ as the continuous linear operator
$$
A_\lambda: L^2((0,T)\times\partial M) \rightarrow L^2(M) \times L^2(\partial M)
$$
so that $^*$ denotes the usual $L^2$-adjoint and $\mf{f}\in L^2(\partial M)$. As the adjoint operators of $\partial^2_t$ and $S$ have been obtained in Lemma~\ref{thm:CalculatingAdjoint}, it remains to compute $W^* c^{-2} W$.

\medskip
We introduce a few more operators before continuing with the analysis. Define the time reversal operator
$$
R: L^2((0,T)\times\partial M) \rightarrow L^2((0,T)\times\partial M), \qquad Rf(t,\cdot):=f(T-t,\cdot) ,\quad 0<t<T
$$
and the low-pass filter 
$$
J: L^2((0,2T)\times\partial M) \rightarrow L^2((0,T)\times\partial M), \qquad Jf(t,\cdot):=\frac{1}{2}\int^{2T-t}_t f(\tau,\cdot) \,d\tau,\quad 0<t<T.
$$
The continuity of $R,J$ are established in the proof of Lemma~\ref{thm:Kstability}.
Moreover, we introduce
$$
\Lambda_T:=P_T \Lambda P^*_T, 
$$
which is the restriction of the hyperbolic ND map $\Lambda$ (originally defined on $L^2((0,2T)\times\partial M)$) onto the subspace $L^2((0,T)\times\partial M)$. The adjoint of $\Lambda_T$ is $\Lambda^\ast_{T} = R \Lambda_{T} R$, as can be easily verified.

With these ingredients, we define a linear operator:
\begin{equation} \label{eq:K}
K: L^2((0,T)\times\partial M) \rightarrow L^2((0,T)\times\partial M), \qquad K = K(\Lambda):= J \Lambda P^\ast_T - R \Lambda_{T} R J P^\ast_T.
\end{equation}
As the composition of continuous operators, $K$ is continuous.
The next lemma, known as the Blagove\u{s}\u{c}enski\u{ı}’s identity~\cite{blagoveshchenskii1967inverse}, shows that $K=W^* c^{-2} W$.

\begin{lemma} \label{thm:CalculatingNormal}
$K=W^* c^{-2} W$ as continuous linear operators on $L^2((0,T)\times\partial M)$.
\end{lemma}

\begin{proof}
As both operators $K$ and $W^* c^{-2} W$ are continuous, it remains to verify that they agree on the dense subset of compactly supported smooth functions. 

For any $f,h\in C^\infty_c ((0,T)\times\partial M)$. Let $u^{f}, u^{h}$ be the (smooth) solutions of \eqref{eq:ibvp} with the Neumann data $f, h$, respectively. Define 
$$
I(t,s) := (u^{f}(t), u^{h}(s))_{c^{-2},M},
$$
then $(Wf, Wh)_{c^{-2},M} = I(T,T)$. We compute
\begin{align*}
 (\partial^2_t - \partial^2_s) I(t,s) & = ((c^2 \Delta_g-q ) u^{f}(t), u^{h}(s))_{c^{-2},M} - (u^{f}(t),(c^2 \Delta_g-q ) u^{h}(s))_{c^{-2},M} \nonumber \\
 & = (\Delta_g u^{f}(t), u^{h}(s))_{M} - (u^{f}(t), \Delta_g u^{h}(s))_{M} \\
 & = (f(t), \Lambda h(s))_{\partial M} - (\Lambda f(t), h(s))_{\partial M},
\end{align*}
where the last equality follows from the Green's identity.
On the other hand, $I(0,s)=\partial_t I(0,s) = 0$ since $u^{f}(0,x)=\partial_t u^{f}(0,x) = 0$. 
Using the relations $f=P^*_T f$ and $h=P^*_T h$, we solve this 1D inhomogeneous wave equation along with the initial conditions to obtain 
\begin{align*}
I(T,T) & = \frac{1}{2} \int^T_0 \int^{2T-t}_{t} \left[ (P^*_T f(t), \Lambda P^*_T h(\sigma))_{\partial M} - (\Lambda P^*_T f(t), P^*_T h(\sigma))_{\partial M} \right] \,d\sigma dt \vspace{1ex}\\
 & = \int^T_0  \left[ (P^*_T f(t) ,  \frac{1}{2}\int^{2T-t}_{t} \Lambda P^*_T h(\sigma) \,d\sigma)_{\partial M} - ( \Lambda P^*_T f(t),  \frac{1}{2}\int^{2T-t}_{t} P^*_T h(\sigma) \,d\sigma)_{\partial M} \right] \,dt \vspace{1ex} \\
 & = (P^*_T f,J \Lambda P^\ast_T h )_{(0,T)\times\partial M} - ( \Lambda P^*_T f, J P^*_T h)_{(0,T)\times\partial M} \\
 & = (f,J \Lambda P^\ast_T h )_{(0,T)\times\partial M} - ( P_T \Lambda P^*_T f, J P^*_T h)_{(0,T)\times\partial M} \\
 & = (f,J \Lambda P^\ast_T h )_{(0,T)\times\partial M} - ( \Lambda_T f, J P^*_T h)_{(0,T)\times\partial M} \\
 & = (f,J \Lambda P^\ast_T h )_{(0,T)\times\partial M} - ( f, R \Lambda_T R J P^*_T h)_{(0,T)\times\partial M} 
\end{align*}
where the last equality holds since $\Lambda^*_T = R \Lambda_T R$. We thus have shown
$$
(Wf, Wh)_{c^{-2},M} = I(T,T) = (f,Kh)_{(0,T)\times\partial M}
$$
for any $f,h\in C^\infty_c ((0,T)\times\partial M)$. 
As $(Wf, Wh)_{c^{-2},M} = (f,W^* c^{-2} Wh)_{(0,T)\times\partial M}$, the proof is complete.
\end{proof}

\begin{remark} \label{thm:SstarJ}
$J$ and $S^*$ agree on the subspace $L^2(M)$. Indeed, for any $\mf{f} \in L^2(\partial M)$, we have
    $$
    S^* \mf{f} = (T-t) \mf{f} = \frac{1}{2} \int^{2T-t}_t \mf{f} \,dt = J \mf{f}.
    $$
\end{remark}

The following commutivity is an intermediate step in establishing the reconstruction formula.

\begin{lemma} \label{thm:commutivity}
$\Lambda_T \partial^{-2}_t = \partial^{-2}_t \Lambda_T$ on $L^2((0,T)\times\partial M)$.    
\end{lemma}
\begin{proof}
We will show that the two operators agree on the dense subset of compactly supported smooth functions, then the commutivity follows since both operators are continuous in view of Lemma~\ref{thm:operators}.

For any $f\in C^\infty_c((0,T)\times\partial M)$, set $h=\partial^{-2}_t f$. Let $u^{P^*_T h}$ be the solution of the initial boundary value problem~\eqref{eq:ibvp} with the Neumann data $P^*_T h$, then
\begin{align*}
\partial^2_t \Lambda_T h & = \partial^2_t P_T \Lambda P^*_T h = \partial^2_t \left( u^{P^*_T h}|_{(0.T)\times\partial M} \right) = \left( \partial^2_t u^{P^*_T h} \right)|_{(0.T)\times\partial M} = \left( u^{\partial^2_t  {P^*_T h}} \right)|_{(0.T)\times\partial M} \\
 & = u^f |_{(0.T)\times\partial M} = P_T \Lambda P^*_T f = \Lambda_T f.   
\end{align*}
Applying $\partial^{-1}_t$ twice combined with the initial conditions $\Lambda_T h|_{t=0} = \partial_t \Lambda_T h |_{t=0} = 0$ completes the proof.
\end{proof}

Combining the results in this section, the next theorem constructs $\mf{L}_\lambda$ from $\Lambda$ for any non-Neumann eigenvalue $\lambda\in\mathbb{C}$. Here, we denote by $\mathcal{L}(L^2(\partial M); H^{1}(\partial M))$ the Banach space of the continuous linear operators mapping $L^2(\partial M)$ into $H^{1}(\partial M)$.

\begin{thm} \label{thm:ReconFormula}
    If $T>\tau_{max}$ and $\lambda$ is not a Neumann eigenvalue of $\mf{P}$, then the following formula computes the fixed-frequency elliptic ND map $\mf{L}_\lambda$ in terms of the hyperbolic ND map $\Lambda$:
    \begin{equation} \label{eq:reconformula}
    \begin{aligned}
    \mf{L}_\lambda & = \lim_{\alpha\rightarrow 0+} S \Lambda_T (A_\lambda^*A_\lambda + \alpha I)^{-1} S^* \\
    & = \lim_{\alpha\rightarrow 0+} S \Lambda_T [(1+\bar{\lambda} Z^2) K (1+\lambda \partial^{-2}_t) + S^* S + \alpha I]^{-1} S^*
    \end{aligned}
    \end{equation}
    and the convergence is in the strong topology of $\mathcal{L}(L^2(\partial M); H^{1}(\partial M))$. Here, $K=K(\Lambda)$ is computable from $\Lambda$ as in~\eqref{eq:K}.
\end{thm}

\begin{proof}
Applying Lemma~\ref{thm:CalculatingAdjoint} and Lemma~\ref{thm:CalculatingNormal}, we can rewrite the linear equation~\eqref{eq:TikUsingWS} as
$$
[(1+\bar{\lambda} Z^2) K (1+\lambda \partial^{-2}_t) + S^* S + \alpha I] \ddot{f}_{\lambda,\alpha} = S^* \mf{f}.
$$
From this representation, we solve for $\ddot{f}_{\lambda,\alpha} \in L^2((0,T)\times\partial M)$ and insert $f_{\lambda,\alpha} = \partial^{-2}_t \ddot{f}_{\lambda,\alpha}\, \in L^2((0,T)\times\partial M)$ into Theorem~\ref{thm:ReconProcedure}(ii) to obtain
\begin{align*}
    \Lambda f_{\lambda,\alpha}(T) & = \mathcal{T} P_T \Lambda P^*_T f_{\lambda,\alpha} \\
    & = \mathcal{T} \Lambda_T \partial^{-2}_t \ddot{f}_{\alpha,\lambda} \\
    & = \mathcal{T} \partial^{-2}_t  \Lambda_T \ddot{f}_{\alpha,\lambda} \\
    & = S \Lambda_T \ddot{f}_{\alpha,\lambda} \\
    & = S \Lambda_T (A_\lambda^*A_\lambda + \alpha I)^{-1} S^* \mf{f} \\
    & = S \Lambda_T [(1+\bar{\lambda} Z^2) K (1+\lambda \partial^{-2}_t) + S^* S + \alpha I]^{-1} S^* \mf{f} \\
    & \rightarrow \mf{L}_\lambda \mf{f} \text{ in } H^{1}(\partial M) \text{ as } \alpha\rightarrow 0+
\end{align*}
where the third inequality follows from Lemma~\ref{thm:commutivity}.
\end{proof}

Note that the operators $\partial^{-2}_t, Z, S, S^*$ are explicitly defined on $(0,T)\times\partial M$, which do not require any knowledge inside $M$. As a result, Theorem~\ref{thm:ReconFormula} gives an algorithm to numerically compute the elliptic ND map $\mf{L}_\lambda$ \textit{without} knowing the coefficients $(c,g,q)$ in the operator $\mf{P}$. 
We present a few numerical examples in Section~\ref{sec:numerics} to validate this formula.

\bigskip
\section{Stability Estimate with Fixed Regularization Parameter}
\label{sec:stability}

Although the formula in Theorem~\ref{thm:ReconFormula} holds in the limiting case, one has to fix the regularization parameter $\alpha$ in numerical implementation, resulting in an approximate construction. This fixed value can be chosen according to various discrepancy principles (such as the Morozov discrepency principle), with the goal to balance data fidelity and regularization effect. 
In this section, we analyze the stability of the formula in Theorem~\ref{thm:ReconFormula} with a fixed regularization parameter $\alpha=\alpha_0$. For ease of notation, we will hide the dependence on the constant $\lambda$ and write the operators $\mf{L}_\lambda, A_\lambda$ simply as $\mf{L}, A$. 
The approximate elliptic ND map is
\begin{equation}
    \begin{aligned}
    \mf{L}_{\alpha_0} & := S \Lambda_T (A_\lambda^*A_\lambda + \alpha_0 I)^{-1} S^* \\
    & = S \Lambda_T [(1+\bar{\lambda} Z^2) K (1+\lambda \partial^{-2}_t) + S^* S + \alpha_0 I]^{-1} S^*. 
    \end{aligned}
\end{equation}

\begin{lemma} \label{thm:Kstability}
    The following stability estimate holds:
    \[\|\tilde K - K\|_{L^2((0,T)\times\partial M)\to L^2((0,T)\times\partial M)}\leq \sqrt{2} T \|\tilde\Lambda - \Lambda\|_{L^2((0,2T)\times\partial M)\to L^2((0,2T)\times\partial M)}.\]
\end{lemma}
\begin{proof}
Recall that $K = J \Lambda P^\ast_T - R \Lambda_{T} R J P^\ast_T$ where $\Lambda_T = P_T \Lambda P^*_T$ (see~\eqref{eq:K}). 
It is clear that
$$
\|P^*_T\|_{L^2((0,T)\times\partial M)\rightarrow L^2((0,2T)\times\partial M)} = 1,
$$
$$
\|P_T\|_{L^2((0,2T)\times\partial M)\rightarrow L^2((0,T)\times\partial M)} = 1,
$$
$$
\|R\|_{L^2((0,T)\times\partial M)\rightarrow L^2((0,T)\times\partial M)} = 1.
$$
Hence
\begin{equation}\label{eq:boundLambda_T}
\|\tilde\Lambda_{T} - \Lambda_{T}\|_{L^2((0,T)\times\partial M)\to L^2((0,T)\times\partial M)}\leq\|\tilde\Lambda - \Lambda\|_{L^2((0,2T)\times\partial M)\to L^2((0,2T)\times\partial M)}.
\end{equation}
In order to estimate operator $J$, notice that
    \begin{align*}
        \|Jf\|_{L^2((0,T)\times\partial M)}^2=&\int_0^T\int_{\partial M}\left|\frac{1}{2}\int_t^{2T-t}f(\tau,x)\dif\tau\right|^2 \, d\mu_{\partial M}(x) dt\\
        \leq&\frac{1}{4}\int_0^T\int_{\partial M}\left(\int_t^{2T-t}\left|f(\tau,x)\right|\dif\tau\right)^2 \, d\mu_{\partial M}(x) dt\\
        \leq&\frac{1}{4}\int_0^T\int_{\partial M}\left(\int_0^{2T}\left|f(\tau,x)\right|\dif\tau\right)^2\, d\mu_{\partial M}(x) dt \\
        \leq&\frac{1}{4}\int_0^T\int_{\partial M}2T\int_0^{2T}\left|f(\tau,x)\right|^2 \,d\tau d\mu_{\partial M}(x) dt\\
        =&\frac{T^2}{2}\|f\|_{L^2((0,2T)\times\partial M)}^2,
    \end{align*}
we conclude
    \[\|J\|_{L^2((0,2T)\times\partial M)\rightarrow L^2((0,T)\times\partial M)} \leq \frac{T}{\sqrt{2}},\]
Combining all the estimates above, we obtain:
\begin{align*}
    & \qquad \|\tilde K - K \|_{L^2((0,T)\times\partial M)\to L^2((0,T)\times\partial M)} \\
    & \leq \| J (\tilde\Lambda - \Lambda) P^\ast_T \|_{L^2((0,T)\times\partial M)\to L^2((0,T)\times\partial M)} + \| R (\tilde{\Lambda}_{T} - \Lambda_T) R J P^\ast_T \|_{L^2((0,T)\times\partial M)\to L^2((0,T)\times\partial M)} \\
    & \leq \frac{T}{\sqrt{2}} \|\tilde{\Lambda} - \Lambda\|_{L^2((0,2T)\times\partial M)\to L^2((0,2T)\times\partial M)} + \frac{T}{\sqrt{2}} \|\tilde{\Lambda} - \Lambda\|_{L^2((0,2T)\times\partial M)\to L^2((0,2T)\times\partial M)} \\
    & = \sqrt{2}T \|\tilde{\Lambda} - \Lambda\|_{L^2((0,2T)\times\partial M)\to L^2((0,2T)\times\partial M)}.
\end{align*}

\end{proof}

The following result suggests that reconstructing the approximate elliptic ND map with a fixed regularization parameter $\alpha_0$ is stable with respect to local perturbations of the hyperbolic ND map. 

\begin{thm} \label{thm:stability}
    Let $\Lambda$ be a fixed hyperbolic ND map. If $T>\tau_{max}$ and $\lambda$ is not a Neumann eigenvalue of $\mf{P}$, then there exists a number $\delta=\delta(\Lambda)>0$ such that any hyperbolic ND map $\tilde\Lambda$ with $\|\tilde\Lambda - \Lambda\|_{L^2((0,2T)\times\partial M)\rightarrow L^2((0,2T)\times\partial M)} \leq \delta$ satisfies the following stability estimate
    $$
    \| \tilde{\mf{L}}_{\alpha_0} - \mf{L}_{\alpha_0} \|_{L^2(\partial M)\rightarrow L^2(\partial M)} \leq C \|\tilde\Lambda - \Lambda\|_{L^2((0,2T)\times\partial M)\rightarrow L^2((0,2T)\times\partial M)}
    $$
    for some constant $C>0$ that depends on the norms of $\Lambda$, $(A^*A+\alpha_0 I)^{-1}, S$.
\end{thm}

\begin{proof}
Recall that the operator $S$ is independent of hyperbolic ND maps. The difference between two approximate elliptic ND maps is
    \begin{align*}
        \mf{L}_{\alpha_0} - \tilde{\mf{L}}_{\alpha_0} & = S \Lambda_T (A^*A + \alpha_0 I)^{-1} S^* - S \tilde{\Lambda}_T (\tilde{A}^*\tilde{A} + \alpha_0 I)^{-1} S^* \\
        & = S \left[ \Lambda_T - \tilde{\Lambda}_T \right] (A^*A + \alpha_0 I)^{-1} S^* + S \tilde{\Lambda}_T \left[ (A^*A + \alpha_0 I)^{-1} - (\tilde{A}^*\tilde{A} + \alpha_0 I)^{-1} \right] S^*.
    \end{align*}
In view of~\eqref{eq:boundLambda_T}, the first term is bounded by $C_1 \|\tilde{\Lambda}-\Lambda\|$ for some constant $C_1>0$ that depends on the norms of $(A^*A + \alpha_0 I)^{-1}$ and $S$. For the second term, we use the relation $(\tilde{A}^*\tilde{A} + \alpha_0 I) - (A^*A + \alpha_0 I) = \tilde{K} - K$ to obtain
\begin{align}
    \| (A^*A + \alpha_0 I)^{-1} - (\tilde{A}^*\tilde{A} + \alpha_0 I)^{-1} \| & = \| (A^*A + \alpha_0 I)^{-1} \left[ \tilde{K} - K \right] (\tilde{A}^*\tilde{A} + \alpha_0 I)^{-1} \| \nonumber \\
    & \leq \| (A^*A + \alpha_0 I)^{-1} \| \; \sqrt{2} T \|\tilde\Lambda - \Lambda\| \; \|(\tilde{A}^*\tilde{A} + \alpha_0 I)^{-1} \| \label{eq:diffinversebound}
\end{align}
where the last inequality follows from Lemma~\ref{thm:Kstability}. Therefore, the second term in the difference is bounded by $C_2 \|\tilde{\Lambda}-\Lambda\|$ for some constant $C_2>0$ that depends on the norms of $(A^*A + \alpha_0 I)^{-1}, (\tilde{A}^*\tilde{A} + \alpha_0 I)^{-1}, S, \Lambda, \tilde\Lambda$ and $T$.

To remove the dependence of $C_2$ on $\tilde\Lambda$, we choose $\delta<1$ and use the upper bound
$$
\|\tilde{\Lambda}\| \leq \|\tilde{\Lambda} - \Lambda\| + \|\Lambda\| \leq \delta + \|\Lambda\| \leq 1 + \|\Lambda\|
$$
to replace $\|\tilde{\Lambda}\|$.
To remove the dependence of $C_2$ on $(\tilde{A}^*\tilde{A} + \alpha_0 I)^{-1}$, we apply~\eqref{eq:diffinversebound} to obtain
\begin{align*}
    \| (\tilde{A}^*\tilde{A} + \alpha_0 I)^{-1} \| & \leq \| (\tilde{A}^*\tilde{A} + \alpha_0 I)^{-1} - (A^*A + \alpha_0 I)^{-1} \| + \| (A^*A + \alpha_0 I)^{-1} \| \\
    & \leq \sqrt{2} T \delta \| (A^*A + \alpha_0 I)^{-1} \| \; \|(\tilde{A}^*\tilde{A} + \alpha_0 I)^{-1} \| + \| (A^*A + \alpha_0 I)^{-1} \|.
\end{align*}
Choose $\delta<\frac{1}{\sqrt{2} T \| (A^*A + \alpha_0 I)^{-1} \| }$ to get
$$
\| (\tilde{A}^*\tilde{A} + \alpha_0 I)^{-1} \| \leq \frac{\| (A^*A + \alpha_0 I)^{-1} }{1-\sqrt{2} T \delta \| (A^*A + \alpha_0 I)^{-1} \|}.
$$
This is an upper bound that does not depend on the tilde terms, hence can be used to replace $\| (\tilde{A}^*\tilde{A} + \alpha_0 I)^{-1} \|$ in the constant $C_2$.
\end{proof}

\bigskip
\section{Numerical Validation} \label{sec:numerics}

In this section, we present numerical examples in 1D as a proof of concept to justify the formula~\eqref{eq:reconformula}.
The computational domain is $M = [-1,1]$ with $\partial M = \{-1,1\}$ and $T=4$. We will systematically use $[\cdot]$ to represent the matrix obtained by discretizing the operator inside the square parenthesis. The discretization process is detailed in Appendix A, where the discretized operators $[P_T], [\mathcal{T}], [R], [J], [Z], [\partial^{-1}_t], [\partial^{-2}_t], [S], [\Lambda], [\Lambda_T], [\mf{L}_\lambda]$ are explicitly constructed.

The discretized operators $[K]$ and $[A_\lambda^* A_\lambda]$ are constructed using matrix multiplications, according to~\eqref{eq:K} and~\eqref{eq:TikUsingWS} respectively, as
\begin{align*}
    [K] & = [J] [\Lambda] [P_T]^\top - [R] [\Lambda_T] [R] [J] [P_T]^\top, \\
    [A_\lambda^* A_\lambda] & = (I + \bar{\lambda} [Z]^2) [K] (I+\lambda [\partial^{-2}_t]) + [S]^\top [S]
\end{align*}
where $\cdot^\top$ denotes the matrix transpose.
In the setup of the subsequent Experiment 1 (i.e, $\lambda=0$, $c=1$, $q=\frac{1}{x+2}$ and $g$ is Euclidean) with no noise, $[K], [S^*S]$ and $[A^*_0 A_0]$ are $2002\times 2002$ matrices. Their first 50 singular values are plotted in Fig.~\ref{fig:singularvalues}. Note that $[A^*_0 A_0]$ is severely ill-conditioned with extremely low numerical rank, justifying the necessity to include regularization.

\begin{figure}[!ht]
    \centering
    \includegraphics[width = 0.32\linewidth]{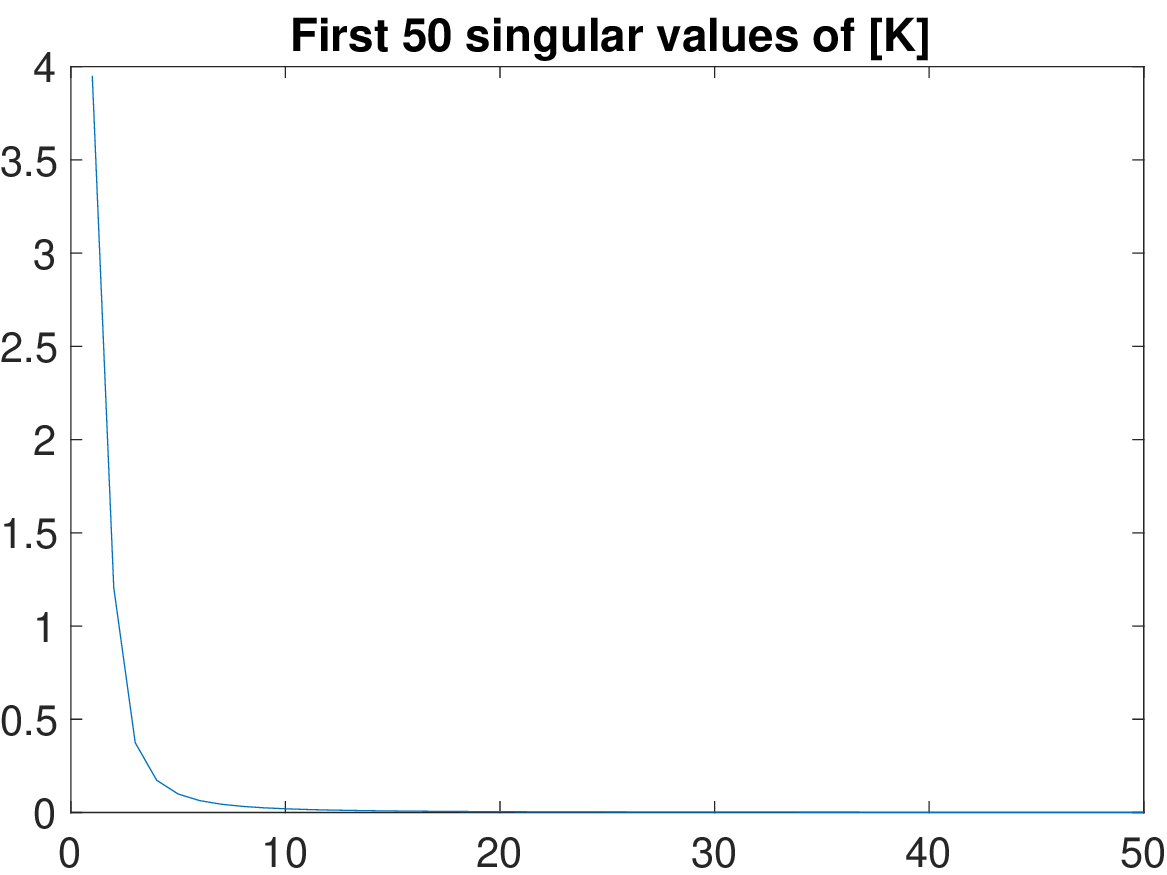}
    \includegraphics[width = 0.32\linewidth]{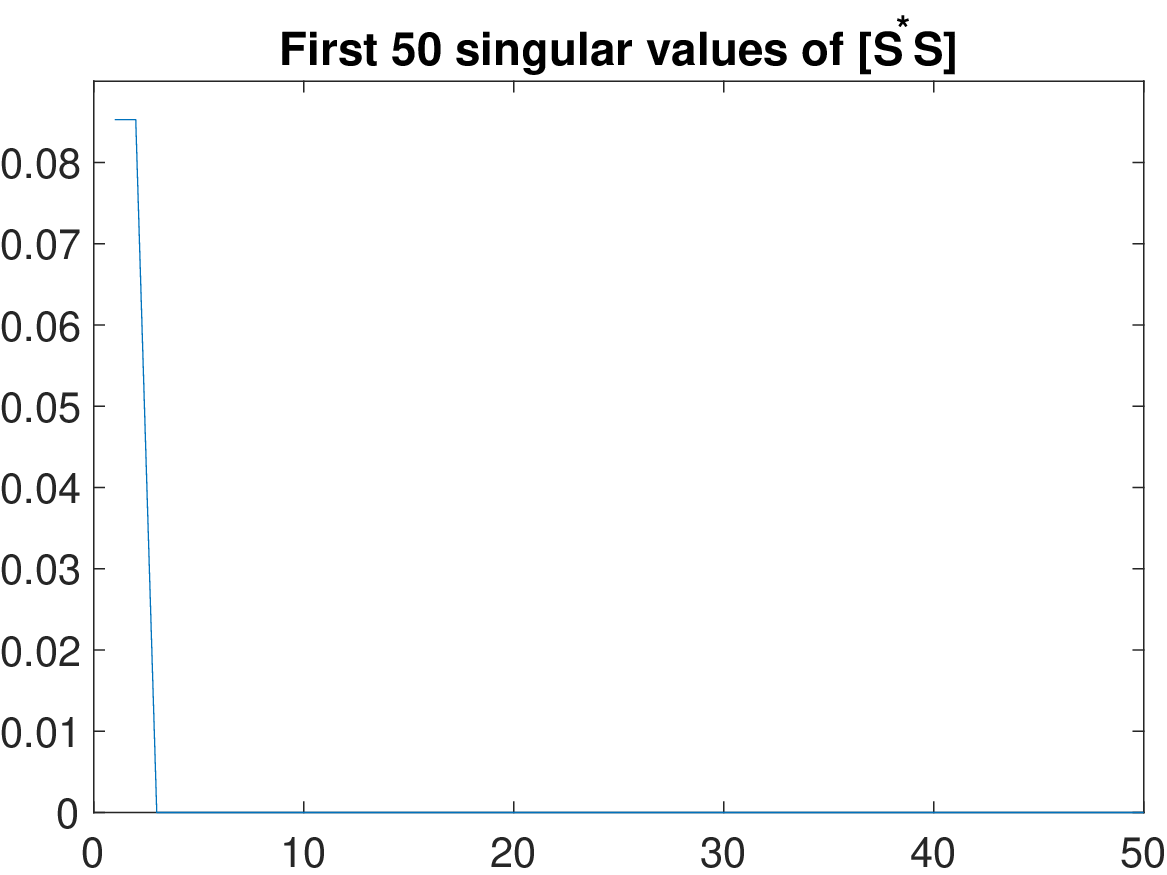}
    \includegraphics[width = 0.32\linewidth]{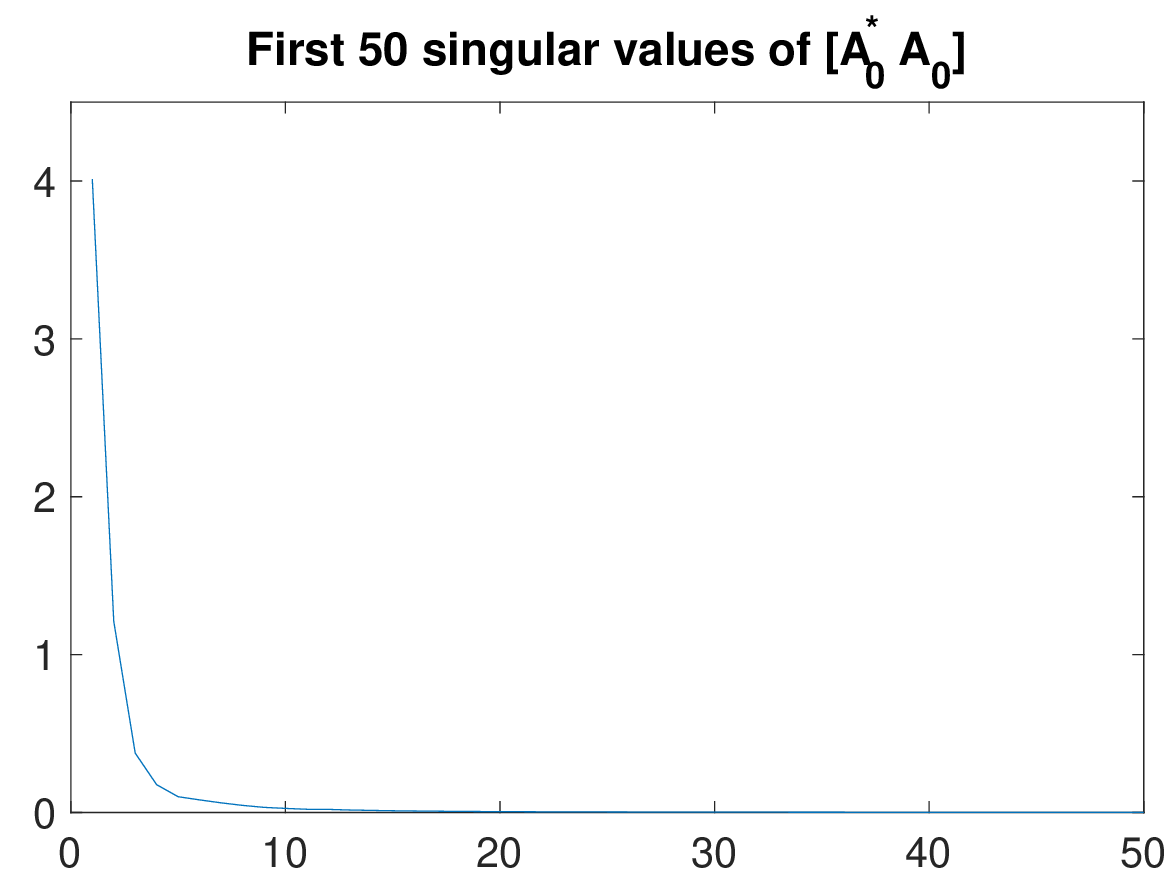}
    \caption{First 50 singular values of $[K], [S^*S]$ and $[A^*_0 A_0]$}
    \label{fig:singularvalues}
\end{figure}

We will validate the stability of the numerical reconstruction by adding random Gaussian noise to each entry of the data $[\Lambda]$. 
We will quantify the error in the reconstructed elliptic ND map $[\tilde{\mf{L}}_\lambda]$ using the \textit{relative Frobenius norm} defined as
$$
\frac{\|[\tilde{\mf{L}}_\lambda] - [\mf{L}_\lambda]\|_F}{\|[\mf{L}_\lambda]\|_F}.
$$
where $\|\cdot\|_F$ denotes the matrix Frobenius norm.
We will quantify the error in the reconstructed elliptic solution $[u^f(T)]$ using the \textit{relative 2-norm} defined as
$$
\frac{\|[u^f(T)] - [\mf{u}_\lambda]\|_2}{\|[\mf{u}_\lambda]\|_2}
$$
where $\|\cdot\|_2$ denotes the vector 2-norm and $[\mf{u}_\lambda]$ is the elliptic solution obtained from an elliptic solver.

\subsection{Experiment 1}
We start with the frequency $\lambda=0$, the constant wave speed $c=1$, the Euclidean metric $g=e$, and the variable potential $q=\frac{1}{x+2}$. 
The regularization parameter is chosen as $\alpha=10^{-4}$.
We impose the elliptic Neumann data $1$ at $x=-1$ and $2$ at $x=2$. The resulting elliptic boundary value problem is
$$
(\Delta-q)\mf{u}_0 = 0\quad\text{ in } (-1,1), \qquad \partial_\nu \mf{u}_0 (-1) = 1, \quad \partial_\nu \mf{u}_0 (1) = 2.
$$
With the choice of $c,g,q$, we have $\tau_{max} = 2 < 4 = T$ and $0$ is clearly not a Neumann eigenvalue of $-\Delta+q$, justifying the assumption in Theorem~\ref{thm:ReconFormula}.

We implement the reconstruction formula~\eqref{eq:reconformula} in the presence of $0\%$ $1\%$ $2\%$ and $5\%$ random Gaussian noise added to $[\Lambda]$. The reconstructions and relative Frobenius error rates are shown in Table~\ref{table:exp1}. Here, the ground-truth elliptic ND map is obtained using the elliptic solver. The experiment validates the reconstruction formula in the Euclidean geometry with and without small random noise.

\begin{table}[!ht]
\scalebox{0.85}{
\begin{NiceTabular}{|c|c|c|c|}
\hline
 & Noise Level & Elliptic ND maps & Relative Frobenius Error Rate \\ \hline
Ground Truth & \diagbox{}{} & $\mf{L}_0 = \begin{bmatrix}
    1.3495 &   0.6534\\
    0.6534 &   1.6640
\end{bmatrix}$ & \diagbox{}{} \\[5pt] \hline 
\multirow{7}{*}{Reconstruction} & $0\%$ & $\tilde{\mf{L}}_0 = \begin{bmatrix}
    1.3371 &   0.6434\\
    0.6436 &   1.6501
\end{bmatrix}$ & $1.00\%$ \\[5pt] \cline{2-4} 
 & $1\%$ & $\tilde{\mf{L}}_0 = \begin{bmatrix}
    1.3526 &   0.6413\\
    0.6494 &   1.6569
\end{bmatrix}$ & $1.95\%$ \\[5pt] \cline{2-4} 
 & $2\%$ & $\tilde{\mf{L}}_0 = \begin{bmatrix}
    1.3412 &   0.6689\\
    0.6545 &   1.7001
\end{bmatrix}$ & $4.31\%$ \\[5pt] \cline{2-4} 
 & $5\%$ & $\tilde{\mf{L}}_0 = \begin{bmatrix}
    1.2264 &   0.5146\\
    0.5993 &   1.5881
\end{bmatrix}$ & $11.4\%$ \\[5pt] \hline
\end{NiceTabular}
}
\caption{Reconstructed ND map in the Euclidean geometry with $0\%,1\%,2\%,5\%$ random Gaussian noise and the corresponding relative Frobenius error rates.}
\label{table:exp1}
\end{table}

We also obtain the boundary control $f$ (which is $f_{\lambda,\alpha}$ in Theorem~\ref{thm:ReconProcedure}) and construct the wave snapshots $u^f(T)$ in the presence of $0\%$ $1\%$ $2\%$ and $5\%$ random Gaussian noise, see Fig.~\ref{fig:exp1_uT}. Here, the ground-truth elliptic solution is obtained using the elliptic solver.
\begin{figure}[!ht]
    \centering
    \includegraphics[width = 0.49\linewidth]{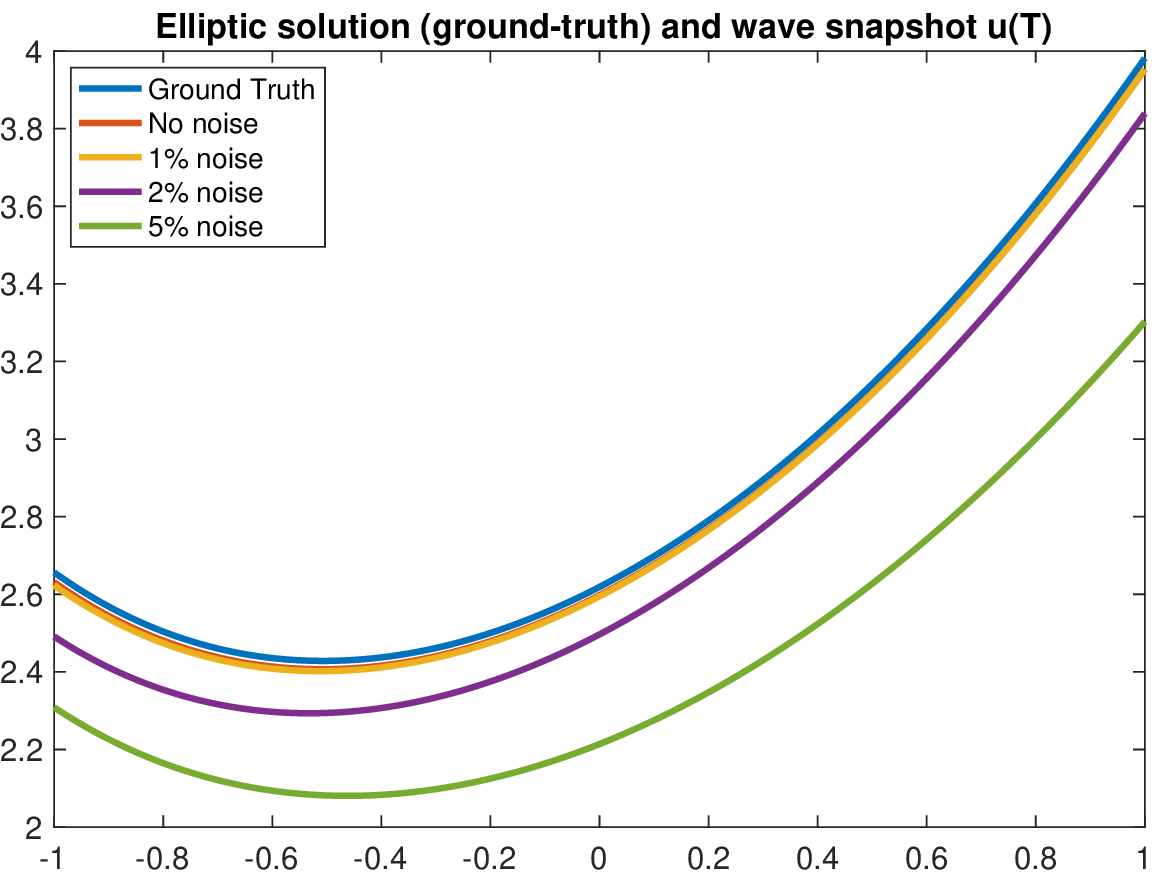}
    \includegraphics[width = 0.49\linewidth]{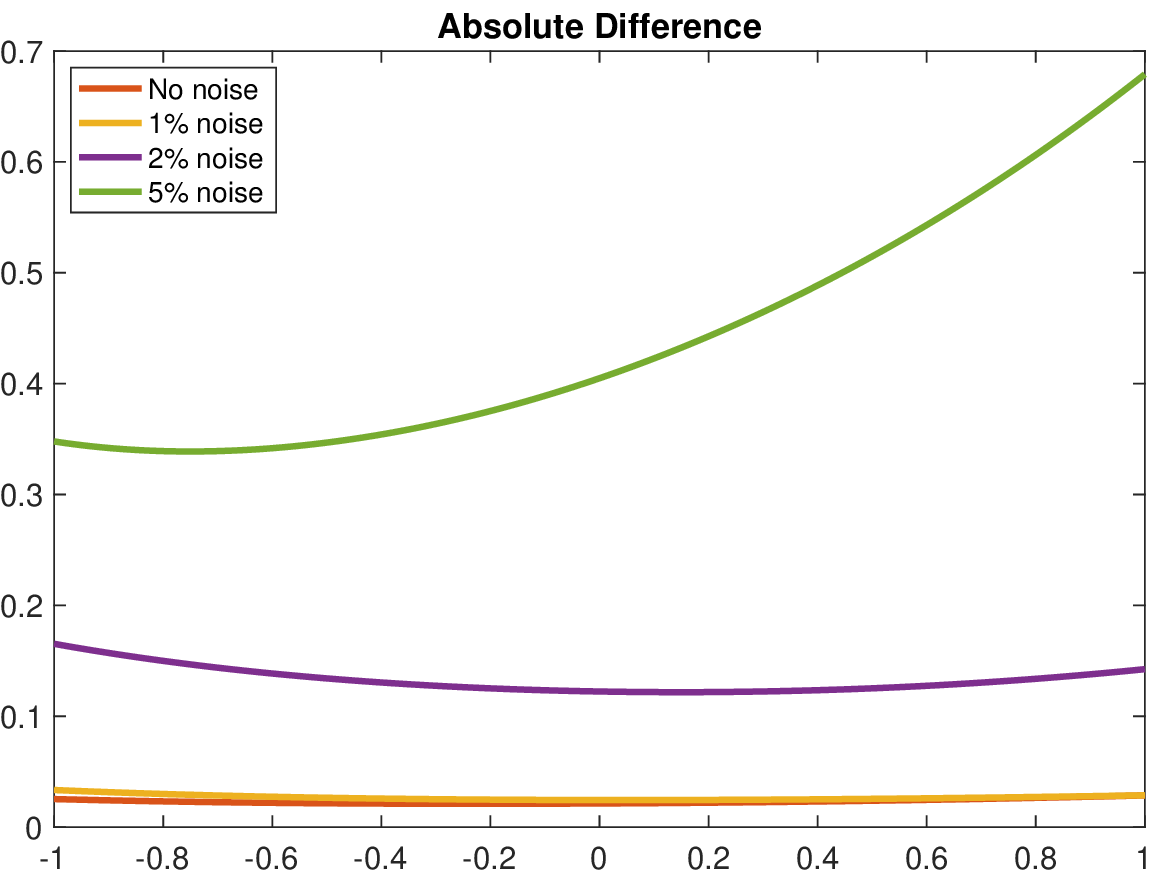}    
    \caption{Left: Elliptic solution $\mf{u}_0$ versus wave snapshot $u^f(T)$ in the Euclidean geometry with $0\%$ $1\%$ $2\%$ and $5\%$ random Gaussian noise. Right: Absolute difference $|[\mf{u}_0] - [u^f(T)]|$.}
    \label{fig:exp1_uT}
\end{figure}

\subsection{Experiment 2}
We continue with the setup as in Experiment 1 with $\lambda=0$, $c=1$, Euclidean metric $e$, $q=\frac{1}{x+2}$, and the same elliptic Neumann data.
This time, no noise is added to $[\Lambda]$, but we gradually decrease the regularization parameter by taking $\alpha = 10^{-1}, 10^{-2}, \dots, 10^{-10}$.
We plot the relative Frobenius error for the reconstructed elliptic ND map $[\tilde{\mf{L}}_0]$ and the relative 2-norm error for the reconstructed elliptic solution $[u^f(T)]$ in Fig.~\ref{fig:changing_alpha}. The experiment validates that the reconstruction formula~\eqref{eq:reconformula} approximates the ground-truth elliptic ND map $[\mf{L}_0]$ as $\alpha\rightarrow 0+$.
\begin{figure}[!ht]
    \centering
    \includegraphics[width = 0.49\linewidth]{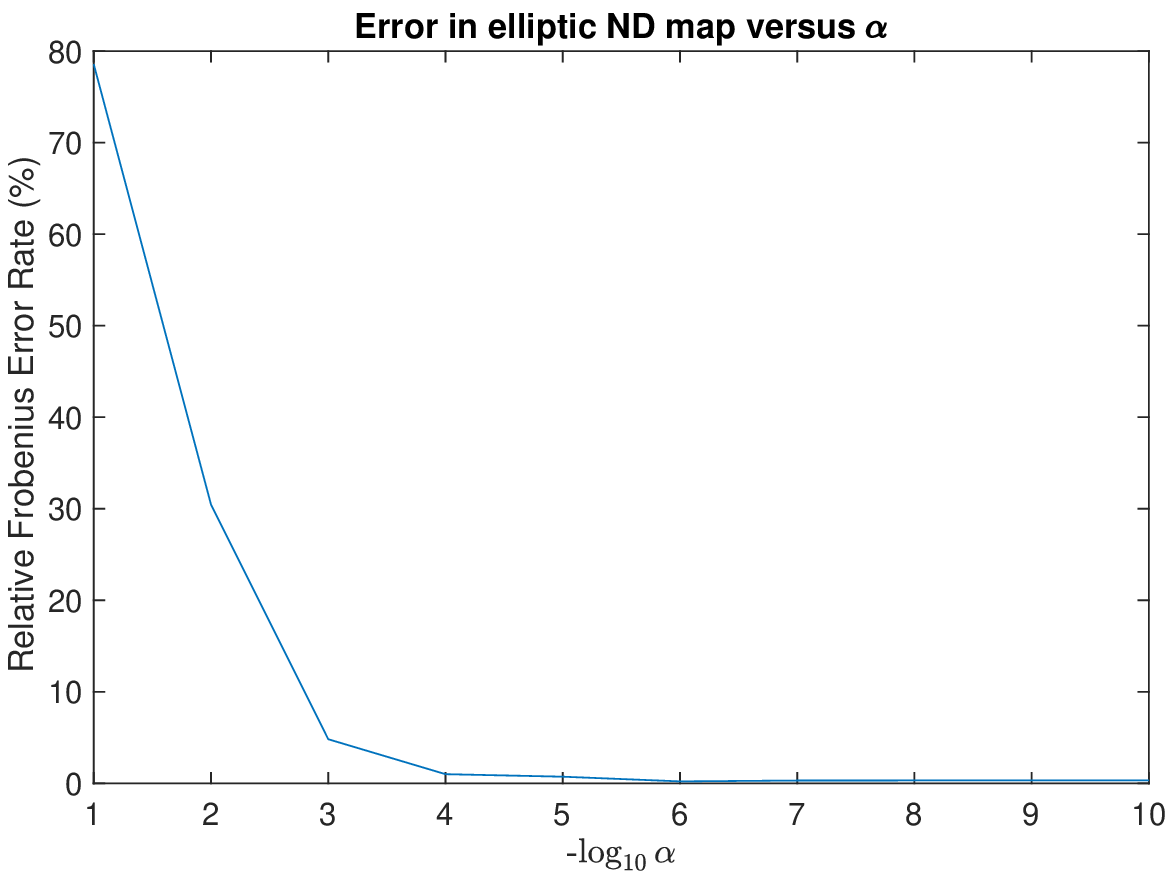}
    \includegraphics[width = 0.49\linewidth]{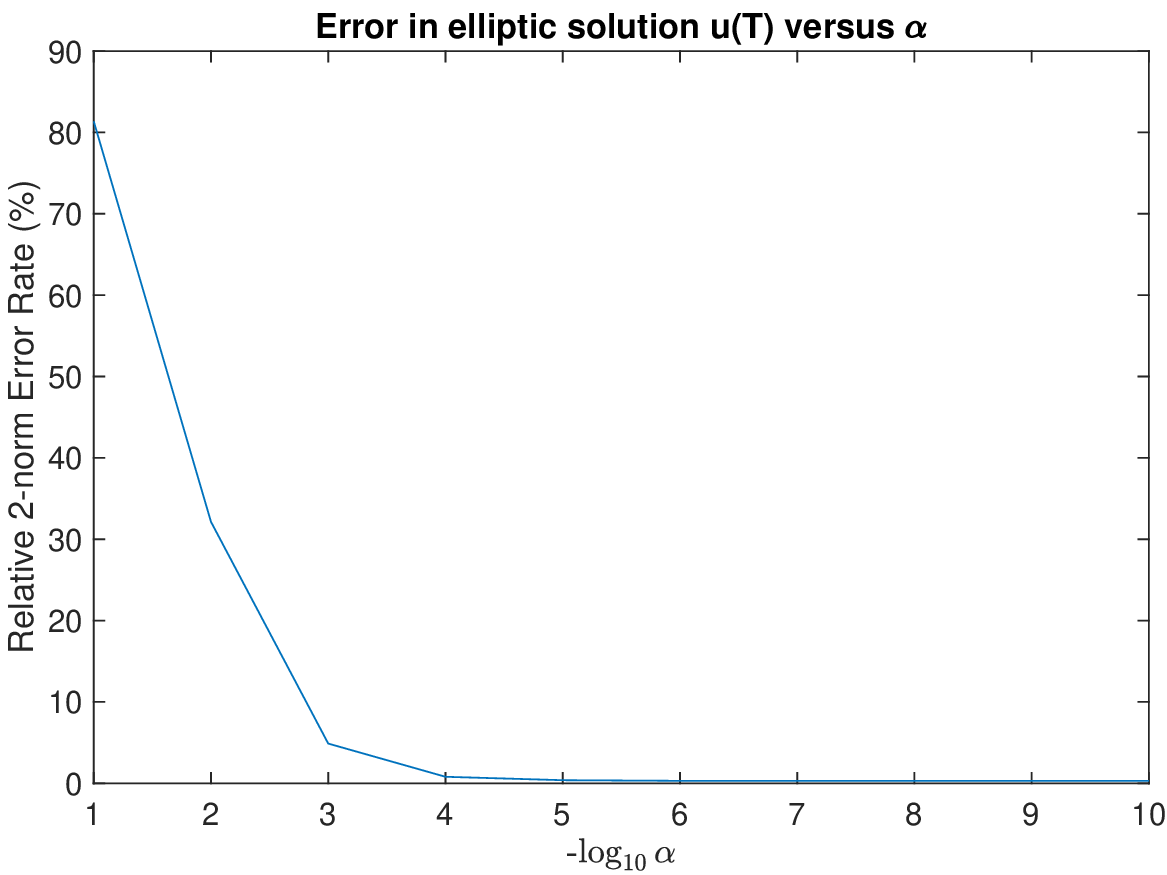}    
    \caption{Left: Relative Frobenius error of the reconstructed elliptic ND map $[\tilde{\mf{L}}_0]$ versus $\alpha = 10^{-1}, 10^{-2}, \dots, 10^{-10}$. Right: Relative 2-norm error of the reconstructed elliptic solution $[\mf{u}_0]$ versus $\alpha = 10^{-1}, 10^{-2}, \dots, 10^{-10}$.}
    \label{fig:changing_alpha}
\end{figure}

\subsection{Experiment 3}
We continue with the frequency $\lambda=0$, the Euclidean metric $g=e$, the variable potential $q=\frac{1}{x+2}$, the regularization parameter $\alpha=10^{-4}$, the same elliptic Neumann data, but take a variable wave speed $c=\cos(\frac{x+1}{2})>0$.
The resulting elliptic boundary value problem is
$$
(c^2 \Delta-q)\mf{u}_0 = 0\quad\text{ in } (-1,1), \qquad \partial_\nu \mf{u}_0 (-1) = 1, \quad \partial_\nu \mf{u}_0 (1) = 2.
$$
In this case, the travel time metric $g':= c^{-2} e$ is conformal to the Euclidean metric $e$, and the elliptic equation can be written as
$$
(c^2 \Delta-q)\mf{u}_0 = \Delta_{g'} \mf{u}_0 - \langle \nabla_{g'} (\ln c), \nabla_{g'}(\mf{u}_0) \rangle_{g'} - q \mf{u}_0 = 0.
$$
where $\nabla_{g'}$ is the gradient operator with respect to the travel time metric $g'$. This is an elliptic equation on $M=[-1,1]$ equipped with the non-Euclidean metric $g'$.
With the choice of $c,g,q$, it is easy to verify that $0$ is not a Neumann eigenvalue of $-c^2 \Delta+q$. Moreover, 
$$
\tau_{max} = \int^1_{-1} \frac{dx}{c(x)} = \int^1_{-1} \sec\left(\frac{x+1}{2}\right) \,dx = 2\ln (\sec 1 + \tan 1) \approx 2.45 < 4 = T,
$$
which justifies the assumption in Theorem~\ref{thm:ReconFormula}.

Again, we implement the reconstruction formula~\eqref{eq:reconformula} in the presence of $0\%$ $1\%$ $2\%$ and $5\%$ random Gaussian noise added to $[\Lambda]$. The reconstructions and relative Frobenius error rates are shown in Table~\ref{table:exp3}. 
We also construct the wave snapshots $u^f(T)$ in the presence of $0\%$ $1\%$ $2\%$ and $5\%$ random Gaussian noise, see Fig.~\ref{fig:exp3_uT}. Here, the ground-truth elliptic ND map and the ground-truth elliptic solution are obtained using the elliptic solver. The experiment validates the reconstruction formula in a non-Euclidean geometry with and without small random noise.

\begin{table}[!ht]
\scalebox{0.85}{
\begin{NiceTabular}{|c|c|c|c|}
\hline
 & Noise Level & Elliptic ND maps & Relative Frobenius Error Rate \\ \hline
Ground Truth & \diagbox{}{} & $\mf{L}_0 = \begin{bmatrix}
    1.2005 &   0.3985\\
    0.3985 &   1.1645
\end{bmatrix}$ & \diagbox{}{} \\[5pt] \hline 
\multirow{7}{*}{Reconstruction} & $0\%$ & $\tilde{\mf{L}}_0 = \begin{bmatrix}
    1.1904 &   0.3913\\
    0.3908 &   1.1528
\end{bmatrix}$ & $1.06\%$ \\[5pt] \cline{2-4} 
 & $1\%$ & $\tilde{\mf{L}}_0 = \begin{bmatrix}
    1.2084 &   0.4070\\
    0.3842 &   1.1480
\end{bmatrix}$ & $1.83\%$ \\[5pt] \cline{2-4} 
 & $2\%$ & $\tilde{\mf{L}}_0 = \begin{bmatrix}
    1.2573 &   0.4147\\
    0.4303 &   1.1759
\end{bmatrix}$ & $4.69\%$ \\[5pt] \cline{2-4} 
 & $5\%$ & $\tilde{\mf{L}}_0 = \begin{bmatrix}
    1.0698 &   0.3433\\
    0.2972 &   1.1224
\end{bmatrix}$ & $11.45\%$ \\[5pt] \hline
\end{NiceTabular}
}
\caption{Reconstructed ND map in a non-Euclidean geometry with $0\%,1\%,2\%,5\%$ random Gaussian noise and the corresponding relative Frobenius error rates.}
\label{table:exp3}
\end{table}

\begin{figure}[!ht]
    \centering
    \includegraphics[width = 0.49\linewidth]{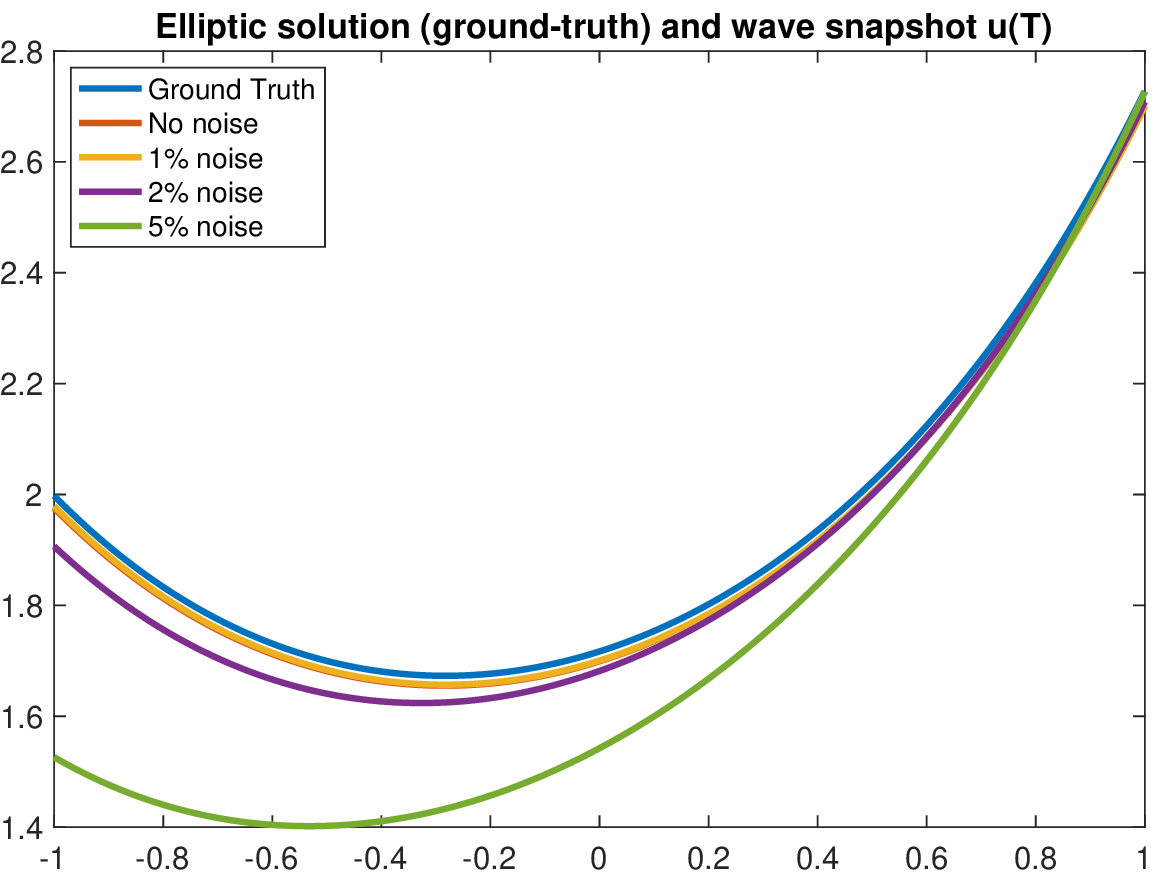}
    \includegraphics[width = 0.49\linewidth]{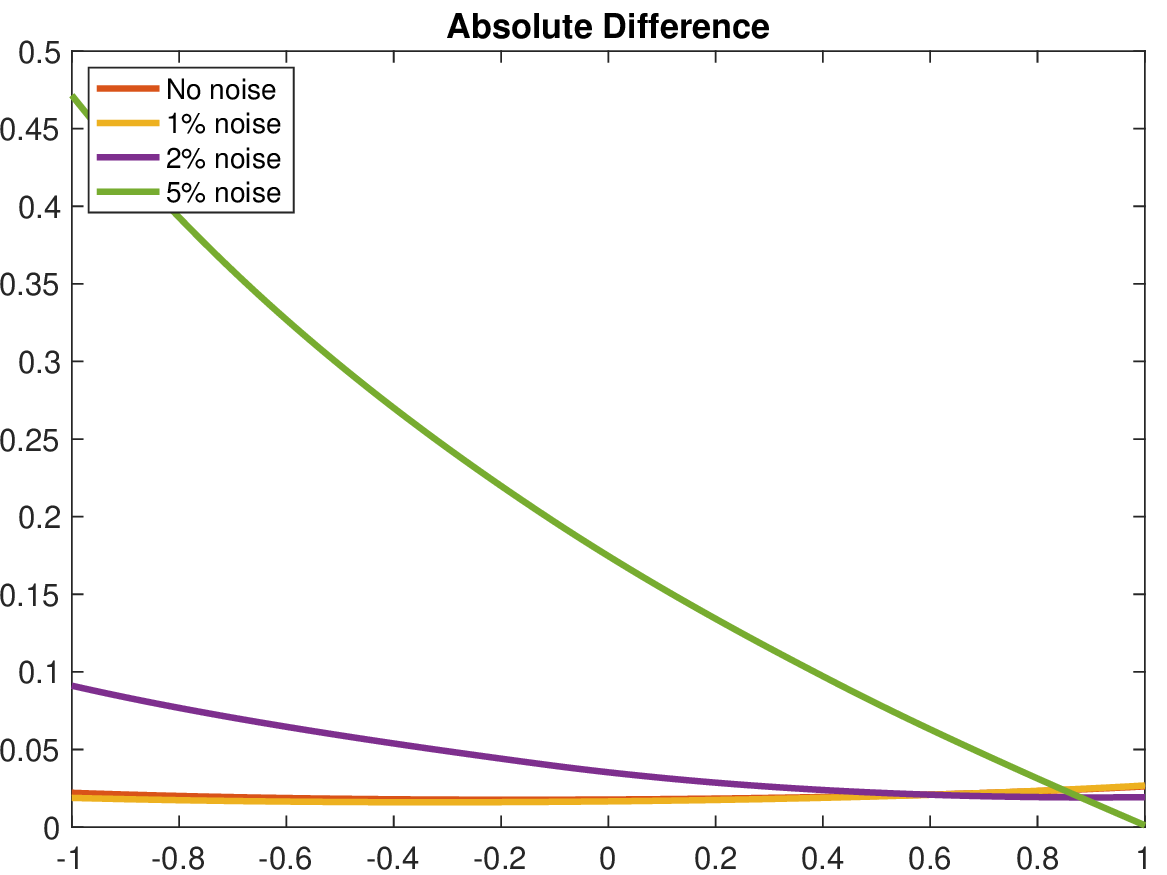}    
    \caption{Left: Elliptic solution $\mf{u}$ versus wave snapshot $u^f(T)$ in a non-Euclidean geometry with $0\%$ $1\%$ $2\%$ and $5\%$ random Gaussian noise. Right: Absolute error $|[\mf{u}_0] - [u^f(T)]|$.}
    \label{fig:exp3_uT}
\end{figure}

\subsection{Experiment 4}
We return to the Euclidean geometry and study the impact of the frequency $\lambda$ to the reconstruction of the elliptic solution $\mf{u}$. Take the constant wave speed $c=1$, the Euclidean metric $g=e$, the constant potential $q=\pi$, and the regularization constant $\alpha=10^{-6}$. Then $\tau_{max} = 2 < 4 = T$, and the resulting Neumann eigenvalues of the operator $-\Delta+q$ are
$$
\lambda_j = \left( \frac{j \pi}{2} \right)^2 + \pi, \qquad j=0,1,2,\dots. 
$$

First, we sweep $\lambda$ across the real interval $(-8,8)$ with step size $0.1$. The relative 2-norm error of the snapshot $u^f(T)$ versus $\lambda$ is shown in Fig.~\ref{fig:exp4_error}. Note that the only Neumann eigenvalues in $(-8,8)$ are
$\lambda_0 = \pi \approx 3.1416$ and $\lambda_1 = \frac{\pi^2}{4} + \pi \approx 5.609$, which are marked by red crosses. It is clear that the reconstruction formula~\eqref{eq:reconformula} remains valid unless the frequency $\lambda$ is close to an eigen-frequency.

Next, we sweep $\lambda$ across the purely imaginary interval $(-8i,8i)$ with step size $0.1i$, where $i$ denotes the imaginary unit. The relative 2-norm error of the snapshot $u^f(T)$ versus $\lambda$ is shown in Fig.~\ref{fig:exp4_error}. Note that there is no Neumann eigenvlaue in $(-8i,8i)$, hence the reconstruction formula~\eqref{eq:reconformula} remains valid. The experiment validates the necessity of the assumption that $\lambda$ is not a Neumann eigenvalue of $-\Delta+q$ in order for the reconstruction formula to be valid.

\begin{figure}[!ht]
    \centering
    \includegraphics[width = 0.49\textwidth]{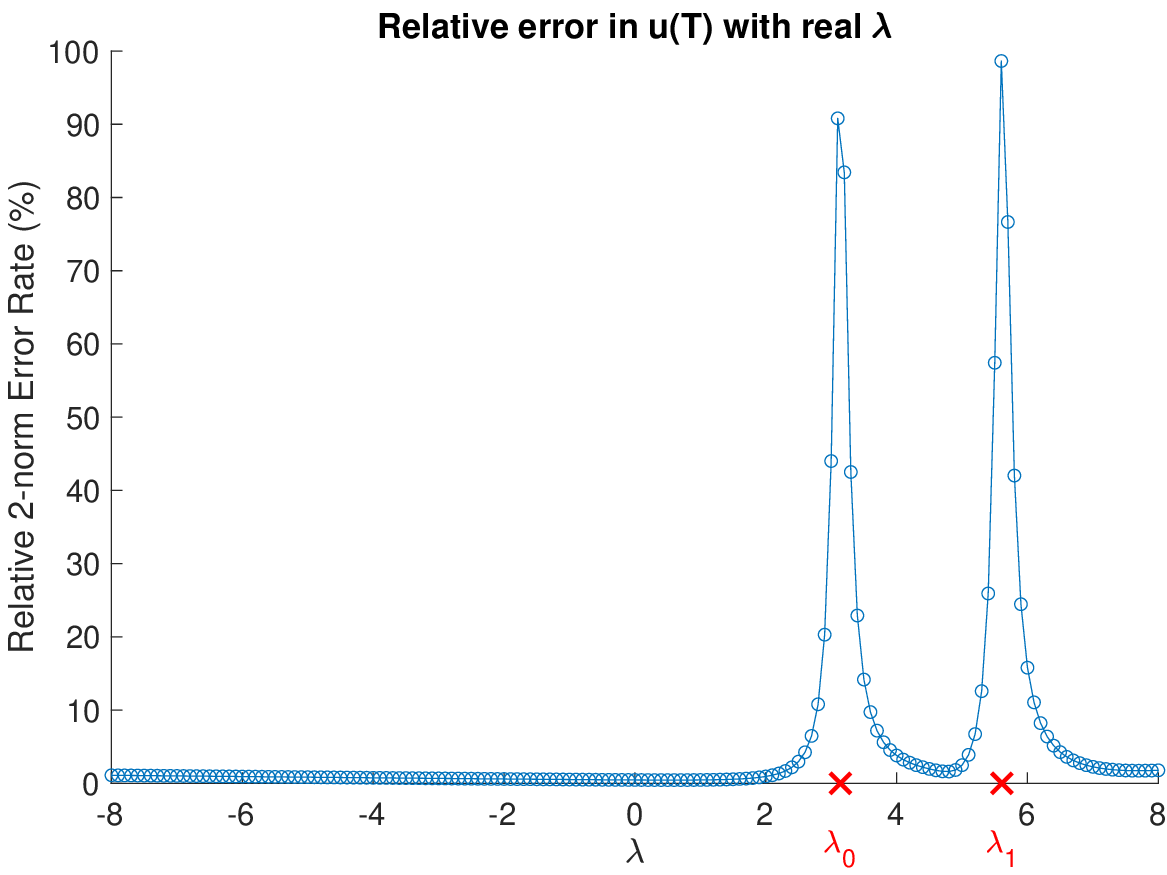}  
    \includegraphics[width = 0.49\textwidth]{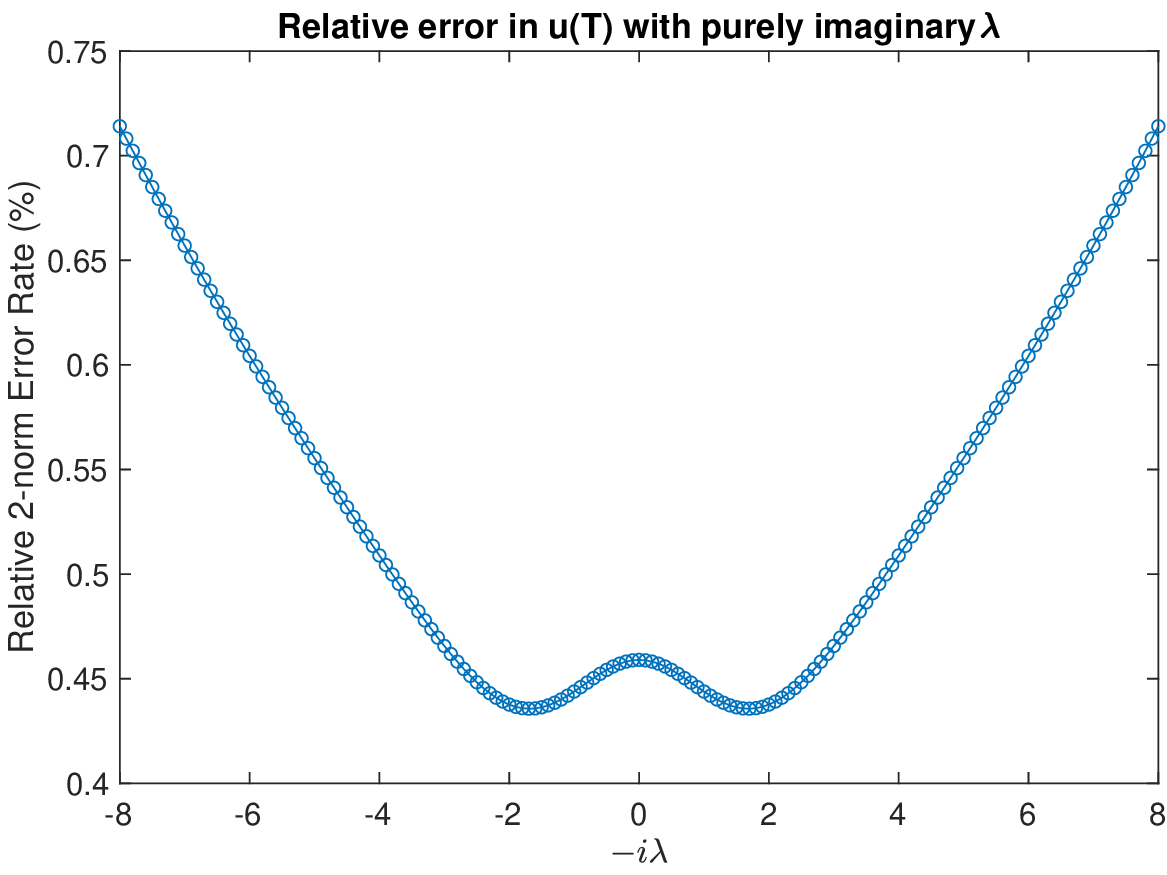}%
    \caption{Left: Error in the wave snapshot $u^f(T)$ (relative to the elliptic solution $\mf{u}$) for real $\lambda\in (-8,8)$. Right: Error in the wave snapshot $u^f(T)$ for purely imaginary $\lambda\in (-8i,8i)$.}
    \label{fig:exp4_error}
\end{figure}

\section*{Acknowledgment}
The research is partially supported by the NSF grants DMS-2237534 and DMS-2220373. The author thanks Dr. Tianyu Yang for the help with some of the numerical experiments.

\bibliographystyle{abbrv}
\bibliography{ref}

\appendix

\section{Details of Discretization}

This appendix details the discretization process for the numerical implementation.

The computational domain is $M = [-1,1]$ with $\partial M = \{-1,1\}$ and $T=4$.
We sample $N_x = 401$ grid points evenly on $M$ with spatial spacing $\Delta x = \frac{1}{200}$. The coordinates of these spatial grid points are 
$$
x_j = -1 + (j-1) \Delta x, \qquad j=1,\dots,N_x.
$$
The temporal spacing is chosen as $\Delta t = \frac{4\Delta x}{5\|c\|_\infty}$ to fulfill the CFL condition for solving the wave equation. Here, $\|c\|_\infty$ denotes the maximum of the wave speed $c$ on $M$. The temporal grid points are 
$$
t_{j} = (j-1) \Delta t, \qquad j=1,\dots,N_t
$$
with $t_{N_t}=T$. Therefore, there are $N_t$ temporal grid points on $[0,T]$, and $2N_t$ temporal grid points on $[0,2T+2\Delta t]$.
Here, we solve the wave equation all the way to $2T+2\Delta t$ instead of $2T$, but the extra data play no essential role in the subsequent experiments, except that some matrices are slightly larger than needed.

In the 1D case, a function on $[0,T]\times\partial M$ consist of two parts. 
Such a function is sampled on the temporal grid points to form a $2N_t\times1$ column vector, where the top half contains the values on $[0,T]\times\{-1\}$ and the bottom half contains the values on $[0,T]\times\{1\}$. Similarly, a function on $[0,2T+2\Delta t]\times\partial M$ is sampled to form a $4N_t\times1$ column vector.

The operators involved are discretized as follows: The discretized projection operator $[P_T]$ is the $2N_t\times4N_t$ matrix:
\[[P_T] = \begin{bmatrix}
    I_{N_t\times N_t}&O_{N_t\times N_t}&&\\
    &&I_{N_t\times N_t}&O_{N_t\times N_t}
\end{bmatrix}\]
where $I_{N_t\times N_t}$ denotes the $N_t\times N_t$ identity matrix and $O_{N_t\times N_t}$ denotes the $N_t\times N_t$ zero matrix. The discretization of the adjoint operator is simply the matrix transpose, i.e., $[P_T^*] = [P_T]^\top$.
The discretized restriction operator $[\mathcal{T}]$ is the $2\times2N_t$ matrix:
\[[\mathcal{T}] = \begin{bmatrix}
    O_{1\times (N_t-1)} & 1 & &\\
    & & O_{1\times (N_t-1)} & 1
\end{bmatrix}\]
The discretized time reversal operator $[R]$ is the $2N_t\times2N_t$ matrix:
\[[R] = \begin{bmatrix}
    &&1&&&\\
    &\iddots&&&&\\
    1&&&&&\\
    &&&&&1\\
    &&&&\iddots&\\
    &&&1&&
\end{bmatrix},\]
which consists of two $N_t\times N_t$ skew-diagonal matrices.

All the integral operators involved are discretized using the trapezoidal rule. For instance, the discretized integral operator $[\partial_t^{-1}]$ is the block-diagonal $2N_t\times2N_t$ matrix
\[[\partial_t^{-1}]=\frac{\Delta t}{2}
\begin{bmatrix}
    \begin{matrix}
    0&&&&&&\\
    1&1&&&&&\\
    1&2&1&&&&\\
    1&2&2&1&&\\
    \vdots&\vdots&\vdots&\vdots&\ddots&\\
    1&2&\dots&\dots&2&1\\
\end{matrix}&\\
&\begin{matrix}
    0&&&&&&\\
    1&1&&&&&\\
    1&2&1&&&&\\
    1&2&2&1&&\\
    \vdots&\vdots&\vdots&\vdots&\ddots&\\
    1&2&\dots&\dots&2&1\\
\end{matrix}
\end{bmatrix},\]
which consists of two identical $N_t\times N_t$ blocks. The discretized second-order integral operator is taken as the square of the first-order matrix, i.e., $[\partial_t^{-2}] = [\partial_t^{-1}]^2$. 
Similarly, the discretized integral operator $[Z]$ is the block-diagonal $2N_t\times2N_t$ matrix with two $N_t\times N_t$ identical diagonal blocks:
\[[Z] = \frac{\Delta t}{2}\begin{bmatrix}
    \begin{matrix}
    1&2&2&\dots&2&1\\
    &1&2&\dots&2&1\\
    &&\ddots&\ddots&\vdots&\vdots\\
    &&&1&2&1\\
    &&&&1&1\\
    &&&&&0\\    
\end{matrix}&\\
&\begin{matrix}
    1&2&2&\dots&2&1\\
    &1&2&\dots&2&1\\
    &&\ddots&\ddots&\vdots&\vdots\\
    &&&1&2&1\\
    &&&&1&1\\
    &&&&&0\\    
\end{matrix}
\end{bmatrix}.\]
The discretized low-pass filter $[J]$ is the $2N_t\times4N_t$ matrix
\[[J] = \frac{\Delta t}{2}\begin{bmatrix}
    \begin{matrix}
    1&2&2&\dots&\dots&2&2&1\\
    &1&2&\dots&\dots&2&1&\\
    &&\ddots&\ddots&\iddots&\iddots&&\\
    &&&1&1&&&
\end{matrix}&\\
&\begin{matrix}
    1&2&2&\dots&\dots&2&2&1\\
    &1&2&\dots&\dots&2&1&\\
    &&\ddots&\ddots&\iddots&\iddots&&\\
    &&&1&1&&&
\end{matrix}
\end{bmatrix}\]
which consists of two identical $N_t\times2N_t$ blocks. The discretized operator $S$ is taken as $[S]=[\mathcal{T}][\partial^{-2}_t]$.

The discretized elliptic ND map $[\mf{L}_\lambda]$ is a $2\times 2$ matrix. This matrix is obtained using a second-order finite-difference elliptic solver: we set $1$ as the left Neumann condition and $0$ as the right Neumann condition to numerically compute the Diricilet data, which forms the first column of $[\mf{L}]$. We then set $0$ as the left Neumann condition and $1$ as the right Neumann condition to obtain the second column of $[\mf{L}]$.

The discretized wave ND map $[\Lambda]$ is a $4N_t\times4N_t$ matrix. This matrix is obtained using a second-order finite difference time domain solver. More precisely, we set $1$ as the $j$-th entry of $[f]$ and $0$ as the other entries (recall that $[f]$ is a $4N_t\times1$ column vector) to numerically find the wave solution $[u^f]$ in $(0,2T+2\Delta t)\times M$. The resulting Dirichlet data $[u^f|_{(0,2T+2\Delta t)\times\Omega}]$ is the $i$-th column of $[\Lambda]$.
Note that if we shift the boundary condition from $f(t,x)$ to $h(t,x) := f(t-t_0,x)$, then the wave solution becomes $u^h(t,x) = u^f(t-t_0,x)$. Using this invariance with respect to temporal shifts, we only need to calculate columns of $[\Lambda]$ for the first few time steps by solving the wave equation. The other columns can be obtained by time-shifting. The truncated wave ND map is calculated by $[\Lambda_T] = [P_T][\Lambda][P_T]^\top$.

\end{document}